\documentclass[a4paper,10pt]{amsart}
\usepackage[backend=bibtex8]{biblatex}
\bibliography{Bibliography}

\usepackage[english]{babel} 
\usepackage[perpage,symbol*]{footmisc} 
\usepackage{esvect} 
\usepackage{subfig} 
\usepackage{placeins} 

\usepackage{amssymb}

\usepackage{array}

\usepackage[colorlinks]{hyperref} 

\usepackage[T1]{fontenc}
\usepackage[utf8]{inputenc} 
\usepackage{csquotes}

\usepackage{pgf,tikz} 
\usetikzlibrary{matrix,arrows,calc}

\title{Poincaré duality for spaces with isolated singularities}

\author{Mathieu Klimczak}
\address{Université des Sciences et Technologies, Lille}
\email{mathieu.klimczak@ed.univ-lille1.fr}

\keywords{Stratified spaces, intersection spaces, Poincaré duality spaces, isolated singularities, rational homotopy}
\subjclass[2010]{55P62, 57P10, 55N33}

\newtheorem{thm}{Theorem}[section]
\newtheorem{defi}{Definition}[thm]
\newtheorem{cor}{Corollary}[thm]
\newtheorem{lem}{Lemma}[thm]
\newtheorem{propo}{Proposition}[thm]
\theoremstyle{definition} \newtheorem{exmp}{Example}[section]
\theoremstyle{remark} \newtheorem{remark}[thm]{Remark}

\newcommand{\tr}[2]{t_{#1}#2} 
\newcommand{\redHI}[3]{\widetilde{H}I_{\overline{#1}}^{#2}(#3)} 
\newcommand{\redhI}[3]{\widetilde{H}I^{\overline{#1}}_{#2}(#3)} 
\newcommand{\IH}[3]{IH_{\overline{#1}}^{#2}(#3)} 
\newcommand{\iH}[3]{IH^{\overline{#1}}_{#2}(#3)} 
\newcommand{\DP}[1]{\mathcal{DP}(#1)} 
\newcommand{\I}[2]{I^{\overline{#1}}#2} 
\newcommand{\Apl}[1]{\mathsf{A}_{\mathsf{PL}}(#1)} 
\newcommand{\cotr}[2]{t^{#1}#2} 
\newcommand{\trL}[1]{t_{\overline{\mathcal{L}}}#1} 
\newcommand{\cotrL}[1]{t^{\overline{\mathcal{L}}}#1} 
\newcommand{\AI}[2]{AI_{\overline{#1}}(#2)} 
\newcommand{\nAI}[2]{A\mathcal{I}_{\overline{#1}}(#2)} 
\newcommand{\nI}[2]{\mathcal{I}^{\overline{#1}}#2} 
\newcommand{\Min}[1]{\mathsf{M}(#1)} 
\newcommand{\col}{\colon\thinspace} 
\newcommand{\im}[1]{\mathrm{im\,}#1} 

\newcommand{\coim}[1]{\mathsf{C}_{#1}} 

\begin{document}

\begin{abstract}  
In this paper we assign, under reasonable hypothesis, to each pseudomanifold with isolated singularities a rational Poincaré duality space. These spaces are constructed with the formalism of intersection spaces defined by Markus Banagl and are indeed related to them in the even dimensional case.
\end{abstract}
\maketitle

\definecolor{cff0000}{RGB}{255,0,0}
\section{Introduction}

We are concerned with rational Poincaré duality for singular spaces. There is at least two ways to restore it in this context :
\begin{itemize}
\item As a self-dual sheaf. This is for instance the case with rational intersection homology.
\item As a spatialization. That is, given a singular space $X$, trying to associate to it a new topological space $X_{DP}$ that satisfies Poincaré duality. This strategy is at the origin of the concept of intersection spaces.
\end{itemize}

Let us briefly recall this two approaches.

While seeking for a theory of characteristic numbers for complex analytic varieties and other singular spaces, Mark Goresky and Robert MacPherson discovered (and then defined in \cite{Goresky1980} for PL pseudomanifolds and in \cite{Goresky1983} for topological pseudomanifolds) a family of groups $\iH{p}{\ast}{X}$ called intersection homology groups of $X$. These groups depend on a multi-index $\overline{p}$ called a perversity. Intersection homology is able to restore Poincaré duality on topological stratified pseudomanifolds. 

If $X$ is a compact oriented pseudomanifold of dimension $n$ and $\overline{p}, \overline{q}$ are two complementary perversities, over $\mathbf{Q}$ we have an isomorphism
\[
\iH{p}{r}{X} \cong \IH{q}{n-r}{X},
\] 
With $\IH{q}{n-r}{X} := \hom(\iH{q}{n-r}{X}, \mathbf{Q})$.

Intersection spaces were defined by Markus Banagl in \cite{Banagl2010} as an attempt to spatialize Poincaré duality for singular spaces. Given a stratified pseudomanifold $X$ of dimension $n$ with only isolated singularities and simply connected links we have a family of topological spaces $\I{p}{X}$ indexed by perversities $\overline{p}$. By analogy with intersection homology, denote by $\redhI{p}{\ast}{X} := \widetilde{H}_{\ast}(\I{p}{X})$ and $\redHI{p}{\ast}{X} := \widetilde{H}^{\ast}(\I{p}{X})$. Over $\mathbf{Q}$ and for complementary perversities $\overline{p}$, $\overline{q}$, we have an isomorphism
\[
\redhI{p}{r}{X} \cong \redHI{q}{n-r}{X}.
\]
One may regard the theory of intersection spaces as an enrichment of intersection homology and we recover informations about intersection homology thanks to those intersection spaces. In particular they have the same information when it comes to the signature of the intersection form as shown in \cite[theorem 2.28]{Banagl2010}. Suppose $X$ is a compact oriented pseudomanifold of dimension $n=4s$ with only isolated singularities and simply connected links. Considering the middle perversity intersection space $\I{m}{X}$ gives us the isomorphism
\[
\redhI{m}{r}{X} \cong \redHI{m}{4s-r}{X}.
\]
Then, it is shown that the two bilinear pairings over $\mathbf{Q}$
\[
b_{HI} : \redhI{m}{2s}{X} \times \redhI{m}{2s}{X} \longrightarrow \mathbf{Q}
\]
and 
\[
b_{IH} : \iH{m}{2s}{X} \times \iH{m}{2s}{X} \longrightarrow \mathbf{Q}
\]
have the same Witt element in $W(\mathbf{Q})$.

It must be noticed that the pairing $\redhI{m}{2s}{X} \times \redhI{m}{2s}{X} \rightarrow \mathbf{Q}$ is not realized as the quadratic form associated to the generalized Poincaré duality of the space $\I{m}{X}$. In fact, the space $\I{m}{X}$ does not have a fundamental class. That is we can't express $b_{HI}(x,y)$ as $\langle [\I{m}{X}], x \cup y \rangle$ where $[\I{m}{X}]$ would be a fundamental class of the space $\I{m}{X}$.

What we show is the existence of a rational Poincaré duality space $\DP{X}$. In particular the pairing defined above for the middle perversity intersection space can be realized as a pairing induced by a classical Poincaré duality and we have
\[
b_{\DP{X}}(x,y) = \langle [\DP{X}], x \cup y \rangle.
\]

Basically, we can separate the construction of the space $\DP{X}$ in two cases, whether the pseudomanifold $X$ has only one isolated singularity or more.

Suppose given a compact, connected oriented pseudomanifold $X$ of dimension $n$ with only isolated singularities $\Sigma = \lbrace \sigma_{1}, \dots, \sigma_{\nu} \rbrace$ and denote by $L_{i}$ the link of the singularity $\sigma_{i}$. We also suppose that the $L_{i}$ are all simply connected.

We have the following definition and theorems

\begin{defi}\label{def:Q_Poinc_approx}
Let $X$ be a compact, connected pseudomanifold with only isolated singularities $\Sigma = \lbrace \sigma_{1}, \dots, \sigma_{\nu} \rbrace$. Denote by $\overline{X}$ the normalization of $X$. A \textbf{good} rational Poincaré approximation of $X$ is a topological space $\DP{X}$ such that 
\begin{enumerate}
\item $\DP{X}$ is a rational Poincaré duality space,
\item there is a rational factorization of the inclusion $i \col X_{reg} \rightarrow \overline{X}$ in two maps
\[
X_{reg} \overset{\phi}{\longrightarrow} \DP{X} \overset{\psi}{\longrightarrow} \overline{X}.
\]
That is $\psi_{r} \circ \phi_{r} = i_{r} \col H_{r}(X_{reg}) \rightarrow H_{r}(\overline{X})$ such that
\begin{itemize}
\item If $\dim X = 2s$, then
\begin{enumerate}
\item $\phi_{r} \col H_{r}(X_{reg}) \longrightarrow H_{r}(\DP{X})$ is an isomorphism for $2s-1 > r > s$ and an injection for $r=s$,
\item $\psi_{r} \col H_{r}(\DP{X}) \longrightarrow H_{r}(\overline{X})$ is an isomorphism for $r < s$ and $r=2s$.
\end{enumerate}
\item If $\dim X = 2s+1$, then
\begin{enumerate}
\item $\phi_{r} \col H_{r}(X_{reg}) \longrightarrow H_{r}(\DP{X})$ is an isomorphism for $2s > r > s+1$ and an injection for $r=s+1$,
\item $\psi_{r} \col H_{r}(\DP{X}) \longrightarrow H_{r}(\overline{X})$ is an isomorphism for $r < s$ and $r=2s+1$ and a surjection for $r=s$.
\end{enumerate} 
\end{itemize}
\end{enumerate}

We say that $\DP{X}$ is a \textbf{very good} rational Poincaré approximation of $X$ if
\begin{itemize}
\item when $\dim X = 2s$, then $\phi_{s}$ is also an isomorphism.
\item when $\dim X = 2s+1$, then $\phi_{s+1}$ and $\psi_{s}$ are also isomorphisms.
\end{itemize}
\end{defi}

\begin{thm}[Unique isolated singularity case]\label{thm:main_thm_one_sing}
Let $X$ be a compact, connected oriented normal pseudomanifold of dimension $n$ with one isolated singularity of link $L$ simply connected. There exists a good rational Poincaré approximation $\DP{X}$ of $X$. Moreover if $\dim X \equiv 0 \mod 4$, then the Witt class associated to the intersection form $b_{\DP{X}}$ is the same that the Witt class associated to the middle intersection cohomology of $X$.
\end{thm}

Depending of the dimension of $X$ we use two different types of truncations on the link of the singularity $L$. When $\dim X = 2s$ this is the classical homology truncation defined by Markus Banagl in \cite{Banagl2010}. The odd dimensional case also use the Lagrangian truncation which will be defined in section \ref{subsec:Lag_truncation}.

The odd dimensional case breaks down in two subcases corresponding to the type of space we are dealing with. If the space is a Witt space we can construct a rational Poincaré duality space, and if the space is an L-space we can perform a Lagrangian truncation to get a rational Poincaré duality space. In the even dimensional case we can always construct a rational Poincaré duality space. We have the following theorem.

\begin{thm}[Multiple isolated singularities case]\label{thm:main_thm_mult_sing}
Let $X$ be a compact, connected oriented normal pseudomanifold of dimension $n$ with only isolated singularities $\Sigma = \lbrace \sigma_{1}, \dots, \sigma_{\nu} ; \nu > 1 \rbrace$ of links $L_{i}$ simply connected. Then, 
\begin{enumerate}
\item If $n = 2s$, there exists a good rational Poincaré approximation $\DP{X}$ of $X$. Moreover, if $\dim X \equiv 0 \mod 4$, then the Witt class associated to the intersection form $b_{\DP{X}}$ is the same as the Witt class associated to the middle intersection cohomology of $X$.
\item If $n = 2s+1$ and $X$ is either a Witt space or an L-space then there exists a good rational Poincaré approximation $\DP{X}$ of $X$. Moreover when $X$ is a Witt space $\DP{X}$ is a very good rational Poincaré approximation of $X$.
\end{enumerate}
\end{thm}

The first section of this paper contains known definitions and results we will use. We first recall the definitions of pseudomanifolds, perversities and we give a brief account of rational homotopy theory. The second part is devoted to the theory of homological truncation theory and intersection spaces defined by Markus Banagl in \cite{Banagl2010}, we also give a rational model of the intersection spaces in \ref{propo:rat_model_IpX}. We extend the homological truncation to a Lagrangian truncation in  the third part \ref{subsec:Lag_truncation}. 

The second section is devoted to the construction of the spaces $\DP{X}$. We first recall the notions we use about Poincaré duality. We then completely develop the method of construction, first with a unique isolated singularity and then explain how to modify the results to get the general theorem in the context of multiple isolated singularities.

We finish with a section of examples, the real algebraic varieties, the nodal hypersurfaces and the Thom spaces.

\section{Background, truncations and intersection spaces}

\subsection{Pseudomanifold, Goresky MacPherson perversity and rational homotopy theory}

In this paper we are concerned with stratified pseudomanifold with isolated singularities of dimension $n$. That is a compact Hausdorff topological space $X$ with stratification : $\emptyset \subset \Sigma \subset X_{n}=X$. Where $\Sigma = \lbrace \sigma_{1}, \dots, \sigma_{\nu} \rbrace$ is a finite set of points, the isolated singularities. Each singularity has a cone-like open neighbourhood $\overset{\circ}{c}L(\sigma_{i},X)$ in $X$ where $L(\sigma_{i},X)$ is a topological manifolds of dimension $n-1$. The space $L(\sigma_{i},X)$ is called a link of $\sigma_{i}$ in $X$. We will denote by $L_{i} := L(\sigma_{i},X)$ a link in $X$ of the singularity $\sigma_{i}$ and by 
\[
L(\Sigma,X) := \bigsqcup_{\sigma_{i} \in \Sigma} L_{i}
\]
the disjoint union of the links. The space $L(\Sigma,X)$ is then a disjoint union of topological manifolds of dimension $n-1$. 

Let $X$ be a stratified pseudomanifold with only isolated singularities 
\[
\Sigma = \lbrace \sigma_{1}, \dots, \sigma_{\nu} \rbrace.
\]
Removing a small open neighbourhood of each singularities $\sigma_{i}$ gives us a manifold with boundary $(X_{reg}, \partial X_{reg})$ where the number of connected components of $\partial X_{reg}$ is the number of connected components of $L(\Sigma, X)$. The manifold $X_{reg}$ is called the regular part of $X$.

\begin{defi}
Let $X$ is a compact, connected oriented pseudomanifold of dimension $n$ with only isolated singularities $\Sigma = \lbrace \sigma_{1}, \dots, \sigma_{\nu} \rbrace$. We say that $X$ is a normal pseudomanifold if the link $L_{i}$ of each singularities $\sigma_{i}$ is connected.
\end{defi}

Given $X$ is a compact, connected oriented pseudomanifold of dimension $n$ with only isolated singularities, one can construct its normalization $\overline{X}$ by considering its regular part $X_{reg}$ and by coning off separately each connected components of the boundary $\partial X_{reg}$. The space $\overline{X}$ is then a normal pseudomanifold with only isolated singularities. We have a map $\overline{X} \longrightarrow X$. 

Since any topological pseudomanifold admits a normalization by the above process and since the definition \ref{def:Q_Poinc_approx} of a Poincaré duality approximation space of $X$ involves a map from $\DP{X}$ to its normalization $\overline{X}$, we will always assume for the rest of this paper that the pseudomanifolds considered here are normal pseudomanifolds.

Since we only work with isolated singularities, a perversity is just a number $\overline{p} \in \lbrace 0,1,\dots, \dim X-2 \rbrace$. We denote by $\overline{m}$ and $\overline{n}$ the following perversities
\[
\begin{cases}
\overline{m} & := \lfloor\frac{\dim X}{2}\rfloor - 1, \\
\overline{n} & := \lceil\frac{\dim X}{2}\rceil - 1. \\
\end{cases}
\]

As for what concerns rational homotopy theory we refer to \cite{Felix2008} and \cite{Felix2001}, we will denote by $A^{\ast}(K)$ some rational model of $K$ and by $\Min{K}$ its minimal model. We will also use the two following results. 

\begin{thm}[Rational Hurewicz theorem,\cite{Dyer1972},\cite{Klaus2004}]\label{thm:Qhurewicz}
Let $K$ be a simply connected topological space with $\pi_{i}(K)\otimes \mathbf{Q} =0$ for $1 <i <r$. Then the Hurewicz map induces an isomorphism
\[
\textsc{Hur}_{i} : \pi_{i}(K) \otimes \mathbf{Q} \longrightarrow H_{i}(K)
\]
for $1 \leq i < 2r-1$ and a surjection for $i=2r-1$.
\end{thm}

\begin{thm}[{\cite[theorem 9.11 p.111]{Felix2001}}]\label{thm:Q_cell_model} 
Every simply connected space $K$ is rationally modelled by a CW-complex $\widetilde{K}$ for which the differential in the integral cellular chain complex is identically zero.
\end{thm}

Let's just say that the construction of $\varphi \col \widetilde{K} \rightarrow K$ is made inductively so that $\varphi$ restrict to rational homotopy equivalences over some sub-CW-complex of $K$. The two following remarks follow from the proof.

\begin{remark}
\label{rmk:Q_cell_model}
\begin{enumerate}
\item Let $K_{1}$ and $K_{2}$ be two simply connected CW-complexes such that $K_{1}^{s} = K_{2}^{s}$ for all $s \leq k$. Then $\widetilde{K_{1}}^{s} = \widetilde{K_{2}}^{s}$ for all $s \leq k$ and the map $\varphi^{K_{1}}_{s}$ and $\varphi^{K_{2}}_{s}$ are equal for all $s \leq k$.
\item If $K$ is a CW-complex of dimension $n$ such that for the cellular chain complex $(C_{\ast}(K), \partial_{\ast})$ we have $\ker \partial_{n} =0$. Then $\widetilde{K}$ is a CW-complex of dimension $n-1$.
\end{enumerate}
\end{remark}

\subsection{Homological truncation and intersection spaces}

In \cite{Banagl2010} Markus Banagl constructed, for a given perversity $\overline{p}$, a space called the perversity $\overline{p}$ intersection space of $X$ denoted by $\I{p}{X}$. We briefly recall the construction.

\begin{defi}
\label{def:ntroncation}
Given an integer $k \geq 3$, a (homological) $k$-truncation structure is a quadruple $(K,K/k,h,\tr{k}{K})$, where
\begin{itemize}
\item $K$ is a simply connected CW-complex,
\item $K/k$ is an $k$-dimensional CW-complex with $(K/k)^{k-1} = K^{k-1}$ and such that the group of $k$-cycles of $K/k$ has a basis of cells,
\item $h \col K/k \rightarrow K^{k}$ is the identity on $K^{k-1}$ and a cellular homotopy equivalence rel $K^{k-1}$, and
\item $\tr{k}{K} \subset K/k$ is a subcomplex such that 

\begin{equation}\label{eqn:ntroncation}
H_{r}(\tr{k}{K};\mathbf{Z}) \cong
\begin{cases}
H_{r}(K;\mathbf{Z})      &  r < k,\\
0                        &  r \geq k,
\end{cases}
\end{equation}
and such that $(\tr{k}{K})^{k-1} = K^{k-1}$.
\end{itemize}
\end{defi}

\begin{propo}[\cite{Banagl2010}]
Given any integer $k \geq 3$, every simply connected CW-complex $K$ can be completed to an $k$-truncation structure $(K,K/k,h,\tr{k}{K})$.
\end{propo}

For the construction of the quadruple $(K,K/k,h,\tr{k}{K})$ we send the reader to \cite{Banagl2010}. A similar construction will be explained later in this paper for the Lagrangian truncation.

Given any $k$-truncation structure $(K,K/k,h,\tr{k}{K})$, we have a homotopy class of maps $f : \tr{k}{K} \rightarrow K$ given by the composition of the following maps
\[
\tr{k}{K} \hookrightarrow K/k \overset{h}{\rightarrow} K^{k} \hookrightarrow K
\]
where the maps at the extremities are cellular inclusions.

Let now $X$ be a compact, connected oriented normal pseudomanifold of dimension $n$ with isolated singularities $\Sigma= \lbrace \sigma_{1}, \dots, \sigma_{\nu} \rbrace$ of simply connected links $L_{i}$, the $L_{i}$ are then $(n-1)$-dimensional compact manifolds. 

Given a Goresky MacPherson perversity $\overline{p}$, put $k(\overline{p}) := n-1-\overline{p}$, we apply the $k(\overline{p})$-truncation on each links $L_{i}$ to get a family of CW-complexes $\tr{k(\overline{p})}{L_{i}}$ together with homotopy classes of maps
\[
f_{i} \col \tr{k(\overline{p})}{L_{i}} \rightarrow L_{i}.
\]
We denote by 
\[
\tr{k(\overline{p})}{L}(\Sigma,X) \overset{f}{\longrightarrow} L(\Sigma,X)
\]
the disjoint union of these maps, with $f := \sqcup_{\sigma_{i}} f_{i}$.

We define the two following homotopy cofibers.

First, consider the homotopy cofiber of the map $f_{i}$, which we denote by $\cotr{k(\overline{p})}{L_{i}}$ and call it the $k(\overline{p})$-cotruncation of $L_{i}$. We have maps 
\[
f^{i}  \col L_{i} \longrightarrow \cotr{k(\overline{p})}{L_{i}}
\]
and
\[
H_{r}(\cotr{k(\overline{p})}{L_{i}};\mathbf{Z}) \cong
\begin{cases}
\mathbf{Z}   &  r = 0, \\
0            &  1 \leq r < k(\overline{p}), \\
H_{r}(L_{i}; \mathbf{Z}) &  r \geq k(\overline{p}).
\end{cases}
\]

We then have a family of maps
\[
\partial X_{reg} = L(\Sigma,X) = \bigsqcup_{i} L_{i} \longrightarrow \bigsqcup_{i} \cotr{k(\overline{p})}{L_{i}}.
\]

Then, we define by $\cotr{k(\overline{p})}{L}(\Sigma,X)$ to be the homotopy cofiber of the map $f$.

\begin{defi}\label{def:intersectionspace}
\begin{enumerate}
\item The intersection space $\I{p}{X}$ of the space $X$ is the homotopy pushout of the solid arrows diagram.
\[
\begin{tikzpicture}
\matrix (m)[matrix of math nodes, row sep=2em, column sep=5em, text height=1.5ex, text depth=0.25ex]
{ L(\Sigma,X)                         & X_{reg}       \\
  \cotr{k(\overline{p})}{L}(\Sigma,X) & \I{p}{X} \\};

\path[->]
(m-1-1) edge node[] {} (m-1-2);
\path[dashed,->]
(m-2-1) edge node[] {} (m-2-2);

\path[->]
(m-1-1) edge node[] {} (m-2-1);
\path[dashed,->]
(m-1-2) edge node[] {} (m-2-2);
\end{tikzpicture}
\]

\item The normal intersection space $\nI{p}{X}$ of the space $X$ is the homotopy pushout of the solid arrows diagram.
\[
\begin{tikzpicture}
\matrix (m)[matrix of math nodes, row sep=2em, column sep=5em, text height=1.5ex, text depth=0.25ex]
{ L(\Sigma,X)                                 & X_{reg}       \\
  \bigsqcup_{i} \cotr{k(\overline{p})}{L_{i}} & \nI{p}{X} \\};

\path[->]
(m-1-1) edge node[] {} (m-1-2);
\path[dashed,->]
(m-2-1) edge node[] {} (m-2-2);

\path[->]
(m-1-1) edge node[] {} (m-2-1);
\path[dashed,->]
(m-1-2) edge node[] {} (m-2-2);
\end{tikzpicture}
\]
\end{enumerate}
\end{defi}

When $X$ is normal with only one isolated singularity, there is no difference between the two definitions. Differences may arise only for the first homology group. In the first case, which is the original definition of \cite{Banagl2010}, we have 
\[
H_{1}(\I{p}{X})= H_{1}(X_{reg}) \oplus \mathbf{Q}^{\beta_{0}(\partial X_{reg})-1}
\]
where $\beta_{0}(\partial X_{reg})$ is the number of connected components of $\partial X_{reg}$. For the normal intersection space $\nI{p}{X}$ we have
\[
H_{1}(\nI{p}{X})= H_{1}(X_{reg}).
\]

We now determine rational models of the truncation, cotruncation and the intersection space of $X$. Let $L$ be a simply connected CW-complex of finite dimension and $(\Min{L},d)$ be its unique minimal Sullivan model. 

Let us make some changes which will be useful when working rational models. Recall that $L$ admits a cellular model $\widetilde{L}$ for which the differential in the integral cellular chain complex is identically zero by theorem \ref{thm:Q_cell_model}. Denote by $\varphi \col \widetilde{L} \rightarrow L$ the rational homotopy equivalence given by this theorem. 

Since $(\tr{k}{L})^{k-1} = L^{k-1}$, the first point of the remark \ref{rmk:Q_cell_model} implies that
\[
\widetilde{\tr{k}{L}}^{k-1} = \widetilde{L}^{k-1}.
\]
By definition $\tr{k}{L}$ is a CW-complex of dimension $k$ such that for its cellular chain complex $(C_{\ast}(\tr{k}{L}), \partial_{\ast})$ we have $\ker \partial_{k}=0$. The CW-complex $\widetilde{\tr{k}{L}}$ is of dimension $k-1$ and is equal to $\widetilde{L}^{k-1}$ by the second point of the remark \ref{rmk:Q_cell_model}.

Consider the following diagram
\[
\begin{tikzpicture}
\matrix (m)[matrix of math nodes, row sep=3em, column sep=5em, text height=2ex, text depth=0.25ex]
{ \widetilde{L}^{k-1}  &      &       & \widetilde{L}   \\
  \tr{k}{L}            &  L/k & L^{k} & L\\};

\path[->]
(m-1-1) edge node[auto] {$\mathrm{incl}$} (m-1-4);
\path[->]
(m-2-1) edge node[auto,swap] {$i$} (m-2-2);
\path[->]
(m-2-2) edge node[auto,swap] {$h$} (m-2-3);
\path[->]
(m-2-3) edge node[auto,swap] {$j$} (m-2-4);
\path[->]
(m-1-1) edge node[auto,swap] {$\varphi_{|}$} (m-2-1);
\path[->]
(m-1-4) edge node[auto] {$\varphi$} (m-2-4);
\end{tikzpicture}
\]
where $\varphi_{|}$ is the restriction of $\varphi$ to the $(k-1)$-cellular skeleton of $\widetilde{L}$, which is then a rational homotopy equivalence. Since $h$ is a homotopy equivalence relative to $L^{k-1}$ this diagram is commutative. The fact that $\varphi$ is a rational homotopy equivalences imply that the minimal models of $L$ and $\widetilde{L}$ are isomorphic. The same is true for the minimal models of $\tr{k}{L}$ and $\widetilde{L}^{k-1}$.

\begin{propo}\label{propo:truncation}
For $k \geq 0$. Let $m : (\Min{L},d) \rightarrow (A^{\ast}(L),d)$ be a rational model of $L$. Denote by $\coim{k-1}$ a supplement of 
\[
\ker (d^{k-1} : A^{k-1}(L) \rightarrow A^{k}(L))
\]
and by $I_{k-1}$ be the differential ideal of $A^{\ast}(L)$ generated by $\coim{k-1} \oplus A^{\geq k}(L)$.

A rational model of $\tr{k}{L}$ is given by $A^{\ast}(L)/I_{k-1}$ and a Sullivan representative of $f : \tr{k}{L} \longrightarrow L$ is given by the projection to the equivalence class
\[
A^{\ast}(L) \twoheadrightarrow A^{\ast}(L)/I_{k-1}.
\]
\end{propo}

\begin{proof}
By the discussion above, we have the isomorphism of minimal models
\[
\Min{\varphi} \col (\Min{L},d) \overset{\cong}{\longrightarrow} (\Min{\widetilde{L}},d).
\]
Composing its inverse with the quasi-isomorphism $m$ gives us the following rational model of $L$
\[
M := m \circ \Min{\varphi}^{-1} \col (\Min{\widetilde{L}},d) \longrightarrow (A^{\ast}(L),d).
\]

Consider now the map $\Min{\mathrm{incl}} \col \Min{\widetilde{L}} \longrightarrow \Min{\widetilde{L}^{k-1}}$. By \cite[theorem 2.28 p.66]{Felix2008} there is a relative minimal model
\[
\begin{tikzpicture}
\matrix (m)[matrix of math nodes, row sep=3em, column sep=5em, text height=2ex, text depth=0.25ex]
{ \Min{\widetilde{L}}  &  \Min{\widetilde{L}^{k-1}}  \\
  \Min{\widetilde{L}}  &  \Min{\widetilde{L}} \otimes \wedge V \\};

\path[->]
(m-1-1) edge node[auto] {$\Min{\mathrm{incl}}$} (m-1-2);
\path[->]
(m-2-1) edge node[auto,swap] {$i$} (m-2-2);
\path[->]
(m-1-1) edge node[auto,swap] {=} (m-2-1);
\path[<-]
(m-1-2) edge node[auto] {$g$} (m-2-2);
\end{tikzpicture}
\]
where $g$ is a quasi-isomorphism and $i$ the canonical inclusion. The fact that $H^{\ast}(\Min{\widetilde{L}})=H^{\ast}(\Min{\widetilde{L}^{k-1}})=H^{\ast}(L)$ for $\ast \leq k-1$ implies that the elements of $V$ are either of degree greater than or equal to $k$, or of degree $k-1$ and not in $\ker d^{k-1}$.

Let then $I_{k-1}$ the differential ideal defined in the proposition and consider the following diagram.
\[
\begin{tikzpicture}
\matrix (m)[matrix of math nodes, row sep=3em, column sep=5em, text height=2ex, text depth=0.25ex]
{ \Min{\widetilde{L}} & \Min{\widetilde{L}} \otimes \wedge V   \\
  A^{\ast}(L)         &   A^{\ast}(L)/I_{k-1} \\};

\path[->]
(m-1-1) edge node[auto] {$i$} (m-1-2);
\path[->]
(m-2-1) edge node[auto,swap] {$p$} (m-2-2);
\path[->]
(m-1-1) edge node[auto,swap] {$M$} (m-2-1);
\path[dashed,->]
(m-1-2) edge node[auto] {$\overline{M}$} (m-2-2);
\end{tikzpicture}
\]
Where $p$ is the projection map. We define the map $\overline{M}$ by
\[
\begin{cases}
\overline{M}(a) = [M(a)] & a \in \Min{\widetilde{L}}, \\
\overline{M}(\wedge V) = 0. & \\
\end{cases}
\]
The image of a product is defined by 
\[
\begin{cases}
\overline{M}(ab) = [M(a)M(b)] & a,b \in \Min{\widetilde{L}}, \\
\overline{M}(vw) = 0 = \overline{M}(v)\overline{M}(w), & v,w \in V, \\
\end{cases}
\]
and for all $v \in V$ and all $a \in \Min{\widetilde{L}}$, the degree of $av$ is greater than or equal to $k$, so we define
\[
\overline{M}(av):=0=\overline{M}(a)\overline{M}(v).
\]

This diagram commutes and $\overline{M}$ defines a quasi-isomorphism.
\end{proof}

We now determine a rational model for the intersection spaces and the normal intersection spaces. The cotruncation being a homotopy cofiber, the next lemma follows from \cite[Proposition 13.6]{Felix2001}

\begin{lem}\label{lem:cotruncation}
The $k$-cotruncation of $L$, $\cotr{k}{L}$, being defined as the homotopy cofiber of the map $\tr{k}{L} \rightarrow L$, a rational model is given by 
\[
\mathbf{Q} \oplus \ker p = \mathbf{Q} \oplus I_{k-1}
\]
where $p$ is the map $ p \col A^{\ast}(L) \twoheadrightarrow A^{\ast}(L)/I_{k-1}.$
\end{lem}

\begin{propo}\label{propo:rat_model_IpX}
Let $X$ be a compact, connected oriented pseudomanifold of dimension $n$ with only isolated singularities $\Sigma= \lbrace \sigma_{1}, \dots, \sigma_{\nu} \rbrace$ of simply connected links $L_{i}=L(\sigma_{i},X)$. Let $\overline{p}$ be a Goresky MacPherson perversity and 
\[
\phi \col (A^{\ast}(X_{reg}),d) \rightarrow (A^{\ast}(L(\Sigma,X)),d)
\]
a surjective model of the inclusion $i \col L(\Sigma,X) = \partial X_{reg} \rightarrow X_{reg}$. A rational model of $\I{p}{X}$ is given by
\[
\AI{p}{X} := (A^{\ast}(X_{reg}),d) \oplus_{A^{\ast}(L)}(\mathbf{Q} \oplus I_{k(\overline{p})},d)
\]
\noindent where $(A^{\ast}(X_{reg}),d)$ is a rational model of the regular part of the pseudomanifold and $(\mathbf{Q} \oplus I_{k(\overline{p})},d)$ a rational model of $\cotr{k(\overline{p})}{L(\Sigma,X)}$.
\end{propo}

\begin{proof}
The intersection space $\I{p}{X}$ of the space $X$ is the homotopy pushout of the diagram. 
\[
\begin{tikzpicture}
\matrix (m)[matrix of math nodes, row sep=2em, column sep=5em, text height=1.5ex, text depth=0.25ex]
{ L(\Sigma,X)                            & X_{reg}       \\
  \cotr{k(\overline{p})}{L(\Sigma,X)}    & \I{p}{X} \\};

\path[->]
(m-1-1) edge node[auto] {$i$} (m-1-2);
\path[->]
(m-2-1) edge node[] {} (m-2-2);

\path[->]
(m-1-1) edge node[] {} (m-2-1);
\path[->]
(m-1-2) edge node[] {} (m-2-2);
\end{tikzpicture}
\]

Then applying $\Apl{-}$ we have a diagram of pullback, 
\[
\begin{tikzpicture}
\matrix (m)[matrix of math nodes, row sep=2em, column sep=5em, text height=1.5ex, text depth=0.25ex]
{ \Apl{\I{p}{X}}   & \Apl{\cotr{k(\overline{p})}{L(\Sigma,X)}}       \\
  \Apl{X_{reg}}    & \Apl{L(\Sigma,X)} \\};

\path[->]
(m-1-1) edge node[] {} (m-1-2);
\path[->]
(m-2-1) edge node[auto,swap] {$\Apl{i}$} (m-2-2);

\path[->]
(m-1-1) edge node[] {} (m-2-1);
\path[->]
(m-1-2) edge node[] {} (m-2-2);
\end{tikzpicture}
\]
and then the quasi isomorphism.
\[
\Apl{\I{p}{X}} \simeq \Apl{X_{reg}} \oplus_{\Apl{L(\Sigma,X)}} \Apl{\cotr{k(\overline{p})}{L(\Sigma,X)}}.
\]

Given the rational models of $X_{reg}$, $L(\Sigma,X)$ and $\cotr{k(\overline{p})}{L(\Sigma,X)}$ thanks to the lemma \ref{lem:cotruncation}, we get a map 
\[
(A^{\ast}(X_{reg}),d) \oplus_{A^{\ast}(L(\Sigma,X))}(\mathbf{Q} \oplus I_{k(\overline{p})},d) \rightarrow \Apl{X_{reg}} \oplus_{\Apl{L(\Sigma,X)}} \Apl{\cotr{k(\overline{p})}{L(\Sigma,X)}}.
\]

With the surjective model $\phi : (A^{\ast}(X_{reg}),d) \rightarrow (A^{\ast}(L(\Sigma,X)),d)$, we get a morphism of short exact sequences,
\[
\begin{tikzpicture}
\matrix (m)[matrix of math nodes, row sep=2em, column sep=1em, text height=1.5ex, text depth=0.25ex]
{ \ker \phi      & (A^{\ast}(X_{reg}),d_{X_{reg}}) \oplus_{A^{\ast}(L(\Sigma,X))}(\mathbf{Q} \oplus I_{k(\overline{p})},d)   & (\mathbf{Q} \oplus I_{k(\overline{p})},d)    \\
  \ker \Apl{i} & \Apl{X_{reg}} \oplus_{\Apl{L(\Sigma,X)}} \Apl{\cotr{k(\overline{p})}{L(\Sigma,X)}}                         &  \Apl{\cotr{k(\overline{p})}{L(\Sigma,X)}}           \\};

\path[->]
(m-1-1) edge node[] {} (m-1-2);
\path[->]
(m-1-2) edge node[] {} (m-1-3);

\path[->]
(m-2-1) edge node[] {} (m-2-2);
\path[->]
(m-2-2) edge node[] {} (m-2-3);

\path[->]
(m-1-1) edge node[] {} (m-2-1);
\path[->]
(m-1-2) edge node[] {} (m-2-2);
\path[->]
(m-1-3) edge node[] {} (m-2-3);
\end{tikzpicture}
\]
The result follows from an application of the five lemma.
\end{proof}

\begin{propo}\label{propo:rat_model_nIpX}
Let $X$ be a compact, connected oriented pseudomanifold of dimension $n$ with only isolated singularities $\Sigma= \lbrace \sigma_{1}, \dots, \sigma_{\nu} \rbrace$ of simply connected links $L_{i}=L(\sigma_{i},X)$. Let $\overline{p}$ be a Goresky MacPherson perversity and 
\[
\phi \col (A^{\ast}(X_{reg}),d) \rightarrow (A^{\ast}(L(\Sigma,X)),d)
\]
a surjective model of the inclusion $i \col L(\Sigma,X) = \partial X_{reg} \rightarrow X_{reg}$. A rational model of the normal intersection space $\nI{p}{X}$ is given by
\[
\nAI{p}{X} := (A^{\ast}(X_{reg}),d) \oplus_{A^{\ast}(L)}(\bigoplus_{i} \mathbf{Q} \oplus I_{k(\overline{p},i)},d)
\]
where $(A^{\ast}(X_{reg}),d)$ is a rational model of the regular part of the pseudomanifold and $(\mathbf{Q} \oplus I_{k(\overline{p},i)},d)$ a rational model of $\cotr{k(\overline{p})}{L_{i}}$.
\end{propo}

\begin{proof}
The proof is exactly the same as the previous proposition unless the normal intersection space $\nI{p}{X}$ is the homotopy pushout of the following diagram.
\[
\begin{tikzpicture}
\matrix (m)[matrix of math nodes, row sep=2em, column sep=5em, text height=1.5ex, text depth=0.25ex]
{ L(\Sigma,X)                            & X_{reg}       \\
  \bigsqcup_{i}\cotr{k(\overline{p})}{L_{i}}    & \nI{p}{X} \\};

\path[->]
(m-1-1) edge node[auto] {$i$} (m-1-2);
\path[->]
(m-2-1) edge node[] {} (m-2-2);

\path[->]
(m-1-1) edge node[] {} (m-2-1);
\path[->]
(m-1-2) edge node[] {} (m-2-2);
\end{tikzpicture}
\]
\end{proof}

In the odd dimensional case, there is a class of pseudomanifolds $X$ for which the truncations $\tr{k(\overline{m})}{}$ and $\tr{k(\overline{n})}{}$ will coincide.

\begin{defi}
\label{def:Wittspace}
Let $X$ be a compact, connected oriented pseudomanifold of dimension $n=2s+1$ with only isolated singularities $\Sigma = \lbrace \sigma_{1}, \dots, \sigma_{\nu} \rbrace$ of links $L_{i}$ simply connected. $X$ is a Witt space if $H^{s}(L_{i})=0$ for all $\sigma_{i} \in \Sigma$.
\end{defi}

\begin{exmp}
\begin{enumerate}
\item The suspension of the complex projective space $S \mathbf{C}P^{3}$ is a Witt space since $H^{3}(\mathbf{C}P^{3}) = 0$.
\item The suspension of the complex projective plane $S \mathbf{C}P^{2}$ is not a Witt-space since $H^{2}(\mathbf{C}P^{2}) = \mathbf{Q}$.
\end{enumerate}
\end{exmp}

\subsection{Lagrangian truncation and Lagrangian intersection spaces}
\label{subsec:Lag_truncation}

First, we recall some facts about quadratic spaces that we will need later.

\begin{defi}
A regular quadratic space $(E,b)$ is a vector space of finite dimension $E$ together with a non degenerate bilinear form 
\[
b : E \times E \rightarrow \mathbf{k},
\]
$b$ being either be a symmetric form or an skew-symmetric one.
\end{defi}

\begin{defi}
An isotropic subspace $V$ of $(E,b)$ is a subspace of $E$ such that for all $x \in V$, $q(x)=b(x,x)=0$. If $2 \dim V = \dim E$, $V$ is then called a Lagrangian subspace.
\end{defi}

\begin{thm}[\cite{SeguinsPazzis2010}, Hyperbolic completion]\label{thm:hyperb_completion}
Let $(E,b)$ be a regular quadratic space of dimension $2m$ and suppose that $E$ posses a Lagrangian subspace subspace $V$ of dimension $m$. Then there exist a basis
\[
(a_{1}, \dots, a_{m},a_{1}^{\ast}, \dots, a_{m}^{\ast})
\]
of $E$ such that
\[
\begin{cases}
b(a_{i},a_{j})          & = 0, \\
b(a_{i}^{\ast},a_{j}^{\ast}) & =0, \\
b(a_{i},a_{j}^{\ast})   & = \delta_{ij}.
\end{cases}
\]

In particular $(a_{1}, \dots, a_{m},a_{1}^{\ast}, \dots, a_{m}^{\ast})$ is a basis in the usual sense of $E$. We also call this an hyperbolic basis.
\end{thm}

The spaces generated respectively by $(a_{1}, \dots, a_{m})$ and $(a_{1}^{\ast}, \dots, a_{m}^{\ast})$ are then Lagrangian subspaces.

Consider $K$ as a simply connected $n$-dimensional CW-complex satisfying Poinca\-ré duality with $n=2s$. We denote by $b$ the non degenerate bilinear form induced by the Poincaré duality with $\mathbf{Q}$ coefficients, consider $\dim H^{s}(K) = 2m$ and
\[
b : H^{s}(K) \times  H^{s}(K) \longrightarrow \mathbf{Q}
\]
where $b(x,y) := \langle x \cup y, [K] \rangle$ with $[K] \in H_{2s}(K)$ the fundamental class and $\langle -,-\rangle$ the evaluation form. 

If $b$ is symmetric suppose that $H^{s}(K)$ posses a Lagrangian subspace $V$ of dimension $m$, let then $(a_{1}, \dots , a_{m})$ be a basis of $V$ and thanks to the theorem \ref{thm:hyperb_completion}, complete $(a_{1}, \dots , a_{m})$ into a hyperbolic basis $(a_{,}, \dots, a_{m}, a_{1}^{\ast}, \dots , a_{m}^{\ast})$ of $H^{s}(K)$.

If $b$ is skew-symmetric then there always exists a Lagrangian subspace and thanks to the theorem \ref{thm:hyperb_completion}, there exist a hyperbolic basis $(a_{,}, \dots, a_{m}, a_{1}^{\ast}, \dots , a_{m}^{\ast})$ of $H^{s}(K)$.

Either way, denote by $V$ and $V^{\ast}$ the subspaces respectively generated by 
\[
V := \langle a_{1}, \dots, a_{m} \rangle
\]
and
\[
V^{\ast} := \langle a_{1}^{\ast}, \dots , a_{m}^{\ast} \rangle.
\]
We have $H^{s}(K) = V \oplus V^{\ast}$.

Remark that since $b(a_{i},a_{i}^{\ast})=1$, $a_{i}$ and $a_{i}^{\ast}$ are Poincaré duals to each other. Denote by $\overline{V}$ and $\overline{V^{\ast}}$ the Poincaré duals in $H_{s}(K)$ of respectively $V$ and $V^{\ast}$ and by $(\overline{a_{1}}, \dots , \overline{a_{m}})$ the basis of $\overline{V}$ and by $(\overline{a_{1}^{\ast}}, \dots , \overline{a_{m}^{\ast}})$ the basis of $\overline{V^{\ast}}$. We have the direct sum
\[
H_{s}(K) = \overline{V} \oplus \overline{V^{\ast}}.
\]

Applying theorem \ref{thm:Q_cell_model} to $K$ we have a rational cellular homotopy equivalence $\varphi \col \widetilde{K} \longrightarrow K$ where $(C_{\ast}(\widetilde{K}), 0)$ is the integral cellular chain complex. We perform and define the Lagrangian truncation on $\widetilde{K}$.

Since the differential of $C_{\ast}(\widetilde{K})$ is zero we have
\[
C_{s}(\widetilde{K}) \otimes \mathbf{Q} = H_{s}(\widetilde{K}) \cong H_{n}(K) = \overline{V} \oplus \overline{V^{\ast}}
\]  
By simply connectivity and the Hurewicz theorem we have the isomorphism
\[
\mathrm{Hur}_{s} : \overline{V} \oplus \overline{V^{\ast}} \overset{\cong}{\longrightarrow} \pi_{s}(\widetilde{K}^{s}, \widetilde{K}^{s-1}).
\]

Let 
\[
\lambda_{i}, \lambda_{i}^{\ast} \col S^{s-1} \rightarrow \widetilde{K}^{s-1}
\]
be choices of representatives for the homotopy classes of $d \overline{a_{i}}$, $d \overline{a_{i}^{\ast}}$. Then for $s$-cells $\lbrace t_{i} \rbrace, \lbrace t_{i}^{\ast} \rbrace$ and using $\lambda_{i}, \lambda_{i}^{\ast}$ as attaching maps we define
\[
\widetilde{K}/\mathcal{L} := \widetilde{K}^{s-1} \cup \bigcup_{\lambda_{i}} t_{i} \cup \bigcup_{\lambda_{i}^{\ast}} t_{i}^{\ast}.
\]
We get the cellular homotopy equivalence $h \col \widetilde{K}/\mathcal{L} \rightarrow \widetilde{K}^{s}$ rel $\widetilde{K}^{s-1}$ the same way as the classical spatial homology truncation in \cite[proposition 1.6]{Banagl2010}. 

\begin{defi}
\label{def:Lagtr}
The Lagrangian truncation of the CW-complex $K$ is defined by 
\[
\trL{K} := \widetilde{K}^{s-1} \cup \bigcup_{\lambda_{i}} t_{i}^{\ast}.
\]
\end{defi}

Moreover, we have $h_{s} \col  H_{s}(\widetilde{K}/\mathcal{L}) \cong H_{s}(\widetilde{K}) \cong H_{s}(\widetilde{K}^{s}, \widetilde{K}^{s-1})$. If $(C_{\ast}(\widetilde{K}/\mathcal{L}), \partial)$ denotes the integral cellular chain complex of $\widetilde{K}/\mathcal{L}$ this implies that 
\[
(C_{\leq s}(\widetilde{K}/\mathcal{L}), \partial)\otimes \mathbf{Q} \longrightarrow (C_{s}(\widetilde{K}), 0)\otimes \mathbf{Q}
\]
yields an isomorphism of homology in degree $s$. This implies that 
\[
\partial \col C_{s}(\widetilde{K}/\mathcal{L}) \rightarrow C_{s-1}(\widetilde{K}/\mathcal{L})
\]
is zero.

The comparison map $\trL{K} \rightarrow K$ is defined as the composition of the maps
\[
\trL{K} \hookrightarrow \widetilde{K}/\mathcal{L} \overset{h}{\rightarrow} \widetilde{K}^{s} \hookrightarrow \widetilde{K} \overset{\varphi}{\rightarrow} K 
\]
where the arrows $\hookrightarrow$ denote cellular inclusions, $h$ is a cellular homotopy equivalence rel $\widetilde{K}^{s-1}$ and $\varphi$ a cellular rational homotopy equivalence. We have 

\begin{equation}\label{eqn:lag_troncation}
H_{r}(\trL{K}) \cong
\begin{cases}
H_{r}(K)      & r \leq s-1\\
\overline{V^{\ast}}      & r = s \\
0                        & r > s.
\end{cases}
\end{equation}

We now define, with the help of Lagrangian truncation, the space called the Lagrangian intersection space associated to $X$

Let $X$ a compact, connected oriented normal pseudomanifold of dimension $n$ with only isolated singularities $\Sigma = \lbrace \sigma_{1}, \dots, \sigma_{\nu} \rbrace$ of links $L_{i}$ simply connected. This time we fix the dimension of $X$ to be odd, $n=2s+1$.

Denote by
\[
b_{i} \col H^{s}(L_{i}) \times H^{s}(L_{i}) \longrightarrow \mathbf{Q}
\]
the non degenerate bilinear form induced by Poincaré duality on the links $L_{i}$. Suppose that for all $i$, $H^{s}(L_{i})$ admit a Lagrangian subspace $V_{i}$ with respect to the bilinear form $b_{i}$. To each $L_{i}$ we apply the Lagrangian truncation process to get maps
\[
f_{i} \col \trL{L_{i}} \longrightarrow L_{i}.
\]
Denote by $\cotrL{L_{i}}$ the homotopy cofiber of the map $f_{i}$ and call it the Lagrangian cotruncation of $L_{i}$, we then have a map
\[
f^{i} \col L_{i} \longrightarrow \cotrL{L_{i}}.
\]
And
\[
H_{r}(\cotrL{L_{i}}) \cong
\begin{cases}
\mathbf{Q}               & r = 0,\\
0                        & 1 \leq r < s,\\
\overline{V_{i}}         & r = s, \\
H_{r}(K)                 & s+1 \leq r \leq 2s.\\
\end{cases}
\]

\begin{defi}
\label{def:ILX_hpo}
The Lagrangian intersection space $\nI{\mathcal{L}}{X}$ of the space $X$ is the homotopy pushout of the solid arrows diagram.
\[
\begin{tikzpicture}
\matrix (m)[matrix of math nodes, row sep=2em, column sep=5em, text height=1.5ex, text depth=0.25ex]
{ L(\Sigma,X)                                 & X_{reg}       \\
  \bigsqcup_{i} \cotrL{L_{i}}                 & \nI{\mathcal{L}}{X} \\};

\path[->]
(m-1-1) edge node[] {} (m-1-2);
\path[dashed,->]
(m-2-1) edge node[] {} (m-2-2);

\path[->]
(m-1-1) edge node[] {} (m-2-1);
\path[dashed,->]
(m-1-2) edge node[] {} (m-2-2);
\end{tikzpicture}
\]
\end{defi}

We want to know when we can perform Lagrangian truncation to get Lagrangian intersection spaces. Lets us first define the class of spaces for which it is possible.

\begin{defi}\label{def:Lspace}
Let $X$ be a compact, connected oriented normal pseudomanifold of dimension $2s+1$ with only isolated singularities $\Sigma = \lbrace \sigma_{1}, \dots, \sigma_{\nu} \rbrace$ of links $L_{i}$ simply connected. $X$ is an L-space if $H^{s}(L_{i})$ has a Lagrangian subspace with respect to the non degenerate bilinear form $b_{i}\col H^{s}(L_{i}) \times H^{s}(L_{i}) \rightarrow \mathbf{Q}$ for all $\sigma_{i} \in \Sigma$.
\end{defi}

\begin{exmp}
\begin{enumerate}
\item The suspension of the torus $S T^{2}$ is an L-space since 
\[
H^{1}(L) = H^{1}(T^{2}) = \mathbf{Q} \oplus \mathbf{Q}.
\]
\item The suspension of the complex projective plane $S \mathbf{C}P^{2}$ is not an L-space since $H^{2}(\mathbf{C}P^{2}) = \mathbf{Q}$.
\end{enumerate}
\end{exmp}

We have the following criterion saying when a space $X$ of odd dimension is an L-space and thus when we can perform a Lagrangian truncation. Let $L_{i}$ be a link of a singularity of $X$, $L_{i}$ is a connected compact manifold of even dimension. Consider its non-degenerate bilinear form induced by Poincaré duality and denote it by $b_{i}$ and let $\sigma(b_{i})$ be its reduced signature, $\sigma(b_{i})$ is related to the Pontryagin numbers by the Hirzebruch signature formula. 

\begin{propo}
Let $X$ a compact, connected oriented normal pseudomanifold of dimension $n$ with only isolated singularities $\Sigma = \lbrace \sigma_{1}, \dots, \sigma_{\nu} \rbrace$ of links $L_{i}$ simply connected. Then
\begin{enumerate}
\item If $n = 4s+3$, $X$ is an L-space.
\item If $n = 4s+1$, $X$ is an L-space if and only if $\sigma(b_{i})=0$ for all $i$.
\end{enumerate} 
\end{propo}

\begin{proof}
Suppose first that $\dim L_{i} \equiv 2 \mod 4 $, that is $\dim X = n =4s+3$. In that case the non-degenerate bilinear form induced by Poincaré duality
\[
b_{i} \col H^{2s+1}(L_{i}) \times H^{2s+1}(L_{i}) \longrightarrow \mathbf{Q}
\]
is skew-symmetric due to the graded commutativity. The form $b_{i}$ is then a non-degenerate symplectic form for all $\sigma_{i} \in \Sigma$, we always have a Lagrangian subspace in that case.

Now, if $\dim L_{i} \equiv 0 \mod 4 $, that is $\dim X = 4s+1$. We can't always apply the Lagrangian truncation. The existence of a Lagrangian subspace for $H^{2s}(L_{i})$ is given by the theorem of Sullivan and Barge, see \cite{Barge1976} and \cite{Sullivan1977}, about rational classification of simply connected manifolds. 

If $X$ is an L-space it is clear by definition \ref{def:Lspace} that $\sigma(b_{i})=0$ for all $i$.

On the other hand, let $(\wedge V_{i}, d)$ be a Sullivan model of $L_{i}$ and $p= \lbrace p_{j} \rbrace$ the Pontryagin numbers of $L_{i}$ related to $\sigma(b_{i}) = 0$. Obviously $L_{i}$ realizes the pair $((\wedge V_{i}, d),p)$, then by the Sullivan-Barge theorem this implies the quadratic form on $H^{2s}(L_{i})$ is equivalent over $\mathbf{Q}$ to a quadratic form 
\[
\sum_{k=1}^{m} x_{k}^{2} - \sum_{k'=1}^{m} x_{k'}^{2}.
\]
This quadratic form is then hyperbolic and posses a Lagrangian subspace.
\end{proof}

\section{The construction of Poincaré duality spaces}
\label{sec:Const_DP_space}
\subsection{Poincaré duality}

We recall here the definitions and results about Poincaré duality needed in the rest of the paper.

\begin{defi}
Let $(X,Y)$ be a pair of CW-complex, we say that $(X,Y)$ is a rational Poincaré duality pair of dimension $n$ if :
\begin{enumerate}
\item $\dim_{\mathbf{Q}} H_{r}(X;\mathbf{Q})$ is finite for all $r$,
\item $Y$ is a sub-CW-complex of $X$ with the same property,
\item there exists a class $[x] \in H_{n}(X,Y; \mathbf{Q})$ such that
\[
- \cap [x] : H^{r}(X;\mathbf{Q}) \longrightarrow H_{n - r}(X,Y;\mathbf{Q})
\]
\noindent is an isomorphism. We call $[x]$ an orientation class of $(X,Y)$
\end{enumerate}
\end{defi}

\begin{remark}
Let $(X,Y)$ be a Poincaré duality pair of dimension $n$, then $Y = (Y,\emptyset)$ is a Poincaré duality pair of dimension $n-1$. Indeed, if $[x] \in H_{n}(X,Y)$ is an orientation class, then $[y] = \partial [x] \in H_{n-1}(Y)$ is an orientation class of $Y$. We say that $Y$ is a Poincaré complex of dimension $n-1$. We also say that $(Y,[y])$ is the oriented boundary of $(X,Y,[x])$.
\end{remark}

Let $(X_{1}, Y, [x_{1}])$ and $(X_{2},Y,[x_{2}])$ be two oriented Poincaré duality pairs of dimension $n$ with the same oriented boundary $(Y, [y])$. Let $\hat{X} := X_{1} \cup_{Y} X_{2}$ the CW-complex obtained by glueing $X_{1}$ and $X_{2}$ along their common boundary $Y$.

\begin{thm}[\cite{Browder1972}, Glueing of 2 oriented Poincaré duality pairs]\label{thm:glueing}
Consider the following diagram 
\[
\begin{tikzpicture}
\matrix (m)[matrix of math nodes, row sep=1.5em, column sep=2.5em, text height=1.5ex, text depth=0.25ex]
{                                    & 0                       \\
				   	                 & H_{n}(\hat{X})          \\
H_{n}(X_{1},Y)\oplus H_{n}(X_{2},Y)  & H_{n}(\hat{X},Y)        \\
						             & H_{n-1}(Y) \\};

\path[->]
(m-1-2) edge node[style] {} (m-2-2);
\path[->]
(m-2-2) edge node[auto] {$i$} (m-3-2);
\path[->]
(m-3-2) edge node[auto] {$\partial$} (m-4-2);

\path[->]
(m-3-1) edge node[auto] {$i_{1} \oplus i_{2}$} (m-3-2);
\end{tikzpicture}
\]
and let $[\hat{x}] = i^{-1}(i_{1} \oplus i_{2}([x_{1}],[-x_{2}])) \in H_{n}(\hat{X})$. Then any two of the following conditions imply the third.
\begin{enumerate}
\item $(\hat{X}, [\hat{x}])$ is an oriented Poincaré complex of dimension $n$ without boundary,
\item $(Y, [y])$ is a Poincaré complex of dimension $n-1$ with orientation class $[y] = \partial_{1} [x_{1}] = \partial_{2} [x_{2}]=\partial_{0}[\hat{x}] \in H_{n-1}(Y)$,
\item $(X_{i}, Y, [x_{i}])$ are Poincaré duality pairs of dimension $n$ with orientation classes $[x_{i}] \in H_{n}(X_{i}, Y)$.
\end{enumerate}
\end{thm}

Let $(X,Y,[x])$ be an oriented Poincaré duality pair of dimension $n=4s$, the following diagram
\[
\begin{tikzpicture}
\matrix (m)[matrix of math nodes, row sep=2.5em, column sep=3em, text height=1.5ex, text depth=0.25ex]
{            &     H^{2s}(X,Y)          & H^{2s}(X)         \\
H_{2s}(Y)    &     H_{2s}(X)            & \hom(H_{2s}(X), \mathbf{Q})\\};

\path[->]
(m-1-2) edge node[auto,swap] {$- \cap [x]$} node[auto]{$\cong$} (m-2-2);
\path[->]
(m-1-2) edge node[style] {} (m-1-3);
\path[->]
(m-2-1) edge node[auto] {$i_{\ast}$} (m-2-2);
\path[->]
(m-2-2) edge node[style] {} (m-2-3);
\path[->]
(m-1-3) edge node[auto] {$\cong$} (m-2-3);
\end{tikzpicture}
\]
gives the vector space $H_{2s}(X)$ a symmetric bilinear form of kernel $i_{\ast}(H_{2n}(Y))$ denoted by $b_{X}$. 

\begin{lem}[Novikov]\label{lem:novikov}
Let $(X_{1}, Y, [x_{1}])$ and $(X_{2},Y,[x_{2}])$ 2 oriented Poincaré duality pairs of dimension $n=4s$ with the same oriented boundary $(Y, [y])$. If $(\hat{X}, [\hat{x}])$ is the space obtained by glueing as in theorem \ref{thm:glueing}, then
\[
[b_{\hat{X}}] = [b_{X_{1}}] - [b_{X_{2}}] \text{ in } W(\mathbf{Q}).
\]
\end{lem}

The rational homotopy type of rational Poincaré duality spaces does not depend of the fundamental class. As stated in the following theorem.

\begin{thm}[\cite{Stasheff1983}]\label{thm:stasheff}
Let $H$ be a Poincaré duality algebra of top dimension $n$ and $H^{1}=0$. Let $X$ be a simply connected rational space with $H(X) \cong H$ except $H^{n}(X)=0$. If $Y = X \cup e^{n}$ with $H(Y) \cong H$, then the rational homotopy type of $Y$ is determined by $X$. Moreover, the cell $e^{n}$ is attached by ordinary Whitehead products (not iterated) with respect to some basis of $\pi_{\ast}(X) \otimes \mathbf{Q}$. 
\end{thm}

We refer to \cite[13.e p.175]{Felix2001} for the definition of a Whitehead product.

\subsection{The unique isolated singularity case}
\label{subsec:one_isol_sing}

In this part we prove the following theorem

\begin{thm}[Unique isolated singularity case]
Let $X$ be a compact, connected oriented normal pseudomanifold of dimension $n$ with one isolated singularity of link $L$ simply connected. Then, there exists a good rational Poincaré approximation $\DP{X}$ of $X$. Moreover if $\dim X \equiv 0 \mod 4$, then the Witt class associated to the intersection form $b_{\DP{X}}$ is the same that the Witt class associated to the middle intersection cohomology of $X$.
\end{thm}

\subsubsection{The even dimensional case}
\label{subsubsec:2k_dim_case}

Consider now $X$ a compact, connected oriented normal pseudomanifold of dimension $n=2s$ with one isolated singularity $\sigma$ of simply connected link $L$. Since $\overline{m}=\overline{n}$ we have a well defined intersection space $IX := \I{m}{X} = \I{n}{X}$.

Let $\phi : S^{2s-1} \rightarrow \cotr{}{L}$ be an arbitrary continuous map with $\cotr{}{L} := \cotr{k(\overline{m})}{L}$ the middle cotruncation of the link. We denote by $\cotr{\phi}{L}$ the space obtained as the result of the following homotopy pushout :
\[
\begin{tikzpicture}[injection/.style={right hook->,fill=white, inner sep=2pt}]

\matrix (m)[matrix of math nodes, row sep=3em, column sep=5em, text height=1.5ex, text depth=0.25ex]
{ S^{2s-1}          &     D^{2s}          \\
\tr{}{L}            &     \cotr{\phi}{L}   \\};

\path[injection]
(m-1-1) edge node[auto] {} (m-1-2);
\path[dashed,->]
(m-2-1) edge node[auto] {} (m-2-2);
\path[->]
(m-1-1) edge node[auto,swap] {$\phi$} (m-2-1);
\path[dashed,->]
(m-1-2) edge node[left=2.5em] {HPO} (m-2-2);
\end{tikzpicture}
\]

\begin{lem}\label{lem:condition_PDpair}
$(\cotr{\phi}{L},L)$ is a Poincaré duality pair if and only if $\phi_{2s-1}$ is an isomorphism. Where $\phi_{2s-1}$ is the connecting homomorphism 
\[
\phi_{2s-1} \col H_{2s}(\cotr{\phi}{L},L) \longrightarrow H_{2s-1}(L)
\]
in the long exact sequence of the pair $(\cotr{\phi}{L},L)$ induced by the attaching map $\phi$.
\end{lem}

\begin{proof}
Suppose $(\cotr{\phi}{L},L)$ is a Poincaré duality pair and denote by $[e_{\phi}]$ a choice of orientation class for the pair. By definition we have $\partial_{2s}[e_{\phi}] = [L]$, but $\partial_{2s} = \phi_{2s-1}$.

On the other hand, if $\phi_{2s-1} : H_{2s}(\cotr{\phi}{L},L) \cong H_{2s-1}(L)$ is an isomorphism let us denote by $[e_{\phi}]:= \phi_{2s-1}^{-1}([L])$. We have to check the commutativity of the following square for the different values of $r$
\[
\begin{tikzpicture}
\matrix (m)[matrix of math nodes, row sep=3em, column sep=5em, text height=1.5ex, text depth=0.25ex]
{ H^{r}(\cotr{\phi}{L})          &     H_{2s-r}(\cotr{\phi}{L},L)          \\
H^{r}(L)                        &     H_{2s-1-r}(L)         \\};

\path[->]
(m-1-1) edge node[auto] {$- \cap [e_{\phi}]$} (m-1-2);
\path[->]
(m-2-1) edge node[auto] {$-\cap [L]$} node[auto, swap]{$\cong$} (m-2-2);
\path[->]
(m-1-1) edge node[auto,swap] {$\text{incl}^{\ast}$} (m-2-1);
\path[->]
(m-1-2) edge node[auto] {$\partial_{2s-r}$} (m-2-2);
\end{tikzpicture}
\]
and the fact that this induces an isomorphism on the upper row. Which are straightforward calculations.
\end{proof}

We now look at a condition on $\phi$ to be an isomorphism.

\begin{lem}
$\phi_{2s-1}$ is an isomorphism if and only if 
\[
\textsc{Hur}_{2s-1}([\phi]) \neq 0 
\]
in $H_{2s-1}(\cotr{}{L}) = H_{2s-1}(L) = \mathbf{Q}$.
\end{lem}

Since the pair $(X_{reg}, \partial X_{reg})$, with $\partial X_{reg} = L$, satisfies Poincaré-Lefschetz duality this is a Poincaré duality pair with the same boundary that $(\cotr{\phi}{L},L)$. The space
\[
\DP{X} := X_{reg} \cup_{L} \cotr{\phi}{L}
\]
is then a Poincaré complex of dimension $2s$ whitout boundary and of orientation class given by
\[
[\DP{X}]= i^{-1}(i_{1} \oplus i_{2}([X_{reg},L],[-e_{\phi}]))
\]

We now show the relation between $IX$ and $\DP{X}$.

\begin{propo}\label{propo:equiv_IX_IXphi}
Let $X$ be a compact, connected oriented normal pseudomanifold of dimension $n=2s$ with one isolated singularity of link $L$ simply connected. The space $\DP{X}$ is then rationally homotopy equivalent to $IX \cup e^{2s}$. If moreover $H^{1}(X_{reg})=0$, then $e^{2s}$ is attached by ordinary Whitehead products (not iterated) with respect to some basis of $\pi_{\ast}(IX) \otimes \mathbf{Q}$ and the rational homotopy type of $\DP{X}$ is determined by $IX$.
\end{propo}

\begin{proof}
Consider the following diagram.
\[
\begin{tikzpicture}[injection/.style={right hook->,fill=white, inner sep=2pt}]

\matrix (m)[matrix of math nodes, row sep=2em, column sep=5em, text height=1.5ex, text depth=0.25ex]
{           & L=\partial X_{reg} & X_{reg}         \\
S^{2s-1}    & \cotr{}{L}         & IX        \\
D^{2s}      & \cotr{\phi}{L}     & \DP{X} \\};

\path[->]
(m-1-2) edge node[style] {} (m-2-2);
\path[->]
(m-1-3) edge node[left=2em] {(I) HPO} (m-2-3);
\path[->]
(m-2-1) edge node[auto,swap] {$i_{2}$} (m-3-1);
\path[->]
(m-2-2) edge node[left=2em] {(III) HPO} (m-3-2);
\path[->]
(m-2-3) edge node[left=2.5em] {(II)} (m-3-3);

\path[injection]
(m-1-2) edge node[auto] {$i_{1}$} (m-1-3);
\path[->]
(m-2-1) edge node[auto]{$\phi$} (m-2-2);
\path[->]
(m-2-2) edge node[style] {} (m-2-3);
\path[->]
(m-3-1) edge node[style] {} (m-3-2);
\path[->]
(m-3-2) edge node[style] {} (m-3-3);
\end{tikzpicture}
\]

The square $(I)$ is a homotopy pushout by definition of the construction of the space $IX$, since $i_{1}$ is a cofibration $\DP{X}$ is the homotopy pushout of the diagram $\tr{}{L}_{\phi} \leftarrow \tr{}{L} \leftarrow L \hookrightarrow M$. By the properties of successive homotopy pushouts (see for exemple \cite{Strom2011}) the square $(II)$ is a homotopy pushout. The square $(III)$ is also a homotopy pushout by definition of $\cotr{\phi}{L}$, this implies that the outside square $(II) + (III)$ is a homotopy pushout. We have the commutative square 
\[
\begin{tikzpicture}
\matrix (m)[matrix of math nodes, row sep=2em, column sep=5em, text height=1.5ex, text depth=0.25ex]
{S^{2s-1}    & IX        \\
 D^{2s}      & \DP{X} \\};

\path[->]
(m-1-1) edge node[style] {} (m-1-2);
\path[->]
(m-1-1) edge node[style] {} (m-2-1);
\path[->]
(m-1-2) edge node[style] {} (m-2-2);
\path[->]
(m-2-1) edge node[style] {} (m-2-2);
\end{tikzpicture}
\]
which is then a homotopy pushout.

So we have a rational homotopy equivalence between $\DP{X}$ and $IX \cup e^{2s}$.

Suppose now that we also have $H^{1}(X_{reg})=0$. The space $\DP{X}$ is simply connected and the theorem \ref{thm:stasheff} then tells us how $e^{2s}$ is attached to $IX$.
\end{proof}

In the case of a pseudomanifold of dimension $n=4s$ with isolated singularities, Markus Banagl showed in \cite[theorem 2.28]{Banagl2010} that the intersection form 
\[
b_{HI} \col \redhI{m}{2s}{X} \otimes \redhI{m}{2s}{X} \longrightarrow \mathbf{Q}
\]
has the same Witt element that the Goresky-MacPherson intersection form
\[
b_{IH} \col \iH{m}{2s}{X} \otimes \iH{m}{2s}{X} \longrightarrow \mathbf{Q}.
\]
that is $[b_{HI}]=[b_{IH}] \in W(\mathbf{Q})$, where $W(\mathbf{Q})$ is the Witt group of the rationals.

Applying the lemma \ref{lem:novikov} to the Poincaré duality pair $(\cotr{\phi}{L}, L)$ constructed above shows that this pair is endowed with a symmetric bilinear form of kernel $i_{\ast}(H_{2s}(L)) = H_{2s}(\cotr{\phi}{L})$, that is the form is the zero form. We see that the Witt class of the intersection form $[b_{\DP{X}}]$ of $\DP{X}$ is completely determined by the intersection form of the regular part $(X_{reg}, \partial X_{reg})$. Which is also the case of $IX$ as showed in \cite[theorem 2.28]{Banagl2010}. Therefore we have the following corollary.

\begin{cor}
$\DP{X}$ is a Poincaré duality rational space whose Witt class associated to the intersection form $b_{\DP{X}}$ is the same that the Witt class associated to the middle intersection cohomology of $X$.
\end{cor}

\subsubsection{The odd dimensional case}
\label{subsubsec:2k+1_dim_case}

Consider now $X$ a compact, oriented pseudomanifold of dimension $n=2s+1$ with one isolated singularity $\sigma$ of simply connected link $L$. 

First let us consider that $X$ is a Witt space, that is $H_{s}(L) = 0$. We then have the following proposition.

\begin{propo}\label{propo:2k+1_Witt_case}
Let $X$ be a compact, connected oriented normal pseudomanifold of dimension $n=2s+1$ with one isolated singularity $\sigma$ of simply connected link $L$. Suppose moreover that $X$ is a Witt space. The constructions of the even dimensional case extend to this case and there exists a rational Poincaré approximation $\DP{X}$ of $X$.
\end{propo}

\begin{proof}
We have to show that there exists a cotruncation $\cotr{}{L}$ of the link of the singularity and a map $\phi \in \pi_{2s}(\cotr{}{L})\otimes \mathbf{Q}$ such that the pair $(\cotr{\phi}{L}, L)$ is a Poincaré duality pair. If so, the theorem \ref{thm:glueing} and the proposition \ref{propo:equiv_IX_IXphi} can be applied.

Consider the cotruncation $\cotr{k(\overline{n})}{L}$ given by the upper middle perversity $\overline{n}$. By definition of Witt spaces and of the cotruncation we have
\[
H_{s}(\cotr{k(\overline{n})}{L}) = H_{s}(L) = 0.
\]
So in fact $\cotr{k(\overline{n})}{L}$ is $s$-connected and we have 
\[
\cotr{k(\overline{m})}{L} = \cotr{k(\overline{n})}{L} := \tr{}{L}.
\]

By the rational Hurewicz theorem we have the isomorphism
\[
\pi_{2s}(\tr{}{L}) \otimes \mathbf{Q} \overset{\cong}{\longrightarrow} H_{2s}(\tr{}{L}).
\]

We still denote by $\phi$ the map obtained by this isomorphism, $\phi_{2s}$ as in the lemma \ref{lem:condition_PDpair} is then a isomorphism and the pair $(\cotr{\phi}{L}, L)$ is a Poincaré duality pair.
\end{proof}

Every pseudomanifold $X$ of dimension $n=2s+1$ which is not a Witt space is in fact an L-space due to the following result of Thom.
\begin{lem}[Thom, \cite{Hirzebruch1966}]\label{lem:thom_hyperb_link}
Let $(X,[x])$ be a rational Poincaré complex of dimension $n=4s$ such that $(X,[x])$ is the boundary of a rational Poincaré duality pair $(Y,[y])$. Then $[b_{X}] = 0 \in W(\mathbf{Q})$. 
\end{lem}

The link $L$ of the singularity of $X$ is a manifold of dimension $2s$ and we have the non degenerate bilinear form induced by Poincaré duality.
\[
b_{L} : H^{s}(L) \times H^{s}(L) \longrightarrow \mathbf{Q}
\]

Suppose that $X$ is not a Witt space. By this result of Thom, the Witt class $[b_{L}] \in W(\mathbf{Q})$ of the intersection form associated to $L$ is zero. This implies that $b_{L}$ is hyperbolic and we have the existence of a Lagrangian subspace. 

\begin{cor}
Any compact, connected oriented normal pseudomanifold of dimension $4s+i$, $i=1,3$, with one isolated singularity $\sigma$ of simply connected link $L$ is an L-space.
\end{cor}

Let us then fix $X$ a compact oriented $L$-space of dimension $4s+i$, $i=1,3$, with one isolated singularity $\sigma$ of simply connected link $L$. The link $L$ of the singularity is of dimension $4s+i-1$ and thanks to the lemma \ref{lem:thom_hyperb_link} we have
\[
H^{2s+\frac{i-1}{2}}(L) = V \oplus V^{\ast}
\]
where $V$ is a Lagrangian subspace of dimension $\frac{1}{2}\dim H^{2s+1}(L) := m$.

Suppose $(a_{1}, \dots, a_{m})$ is a basis of $V$ and complete it into a hyperbolic basis using theorem \ref{thm:hyperb_completion}. Apply then the Lagrangian truncation to $L$ and denote by $\cotrL{L}$ the homotopy cofiber of the map $\trL{L} \longrightarrow L$.

Let $(\Min{\cotrL{L}},d)$ be the Sullivan minimal model of this Lagrangian cotruncation and let $\varpi$ be the element of degree $4s+i-1$ in $(\Min{\cotrL{L}},d)$ representing the fundamental class $[L]^{\ast} \in H^{4s+i-1}(L)$.

\begin{propo}
Suppose $\varpi$ is an indecomposable element of $(\Min{\cotrL{L}},d)$, then there exists a map $\phi \in \pi_{4s+i-1}(\cotrL{L})\otimes \mathbf{Q}$ such that $\textsc{Hur}_{4s+i-1}([\phi]) = 1$.
\end{propo}

\begin{proof}
Suppose $\varpi$ is an indecomposable element, that is $\varpi \in W^{4s+i-1}$ where $W = \oplus_{k \geq 0} W^{k}$ is the graded vector space generating $\Min{\cotrL{L}}$. Then by \cite[theorem 15.11]{Felix2001}, we have the natural isomorphism
\[
W^{4s+i-1} \overset{\cong}{\longrightarrow} \hom(\pi_{4s+i-1}(\cotrL{L}), \mathbf{Q}).
\]

That natural isomorphism is the same as saying that the bilinear pairing 
\[
\langle -,- \rangle : W \times \pi_{\ast}(\cotrL{L}) \longrightarrow \mathbf{Q}
\]
is non degenerate. So there is a map $\varphi \in \pi_{4s+i-1}(\cotrL{L})$ such that $\langle \varpi , [\varphi]\rangle \neq 0$.

Recall that 
\[
m_{\cotrL{L}} : (\wedge W, d) = (\Min{\cotrL{L}},d) \longrightarrow \Apl{\cotrL{L}}
\]
denotes the minimal Sullivan model of $\cotrL{L}$ and denote by 
\[
\textsc{Hur}_{k} \col \pi_{k}(\cotrL{L})\otimes \mathbf{Q} \longrightarrow H_{k}(\cotrL{L})
\]
the Hurewicz map. Since $\im{d} \subset \wedge^{\geq 2} W$, quotienting by $\wedge^{\geq 2} W$ defines a linear map $\xi \col H^{+}(\wedge W) \rightarrow W$, since $\varpi$ represents the fundamental class $[L]^{\ast} \in H^{4s+i-1}(L) = H^{4s+i-1}(\cotrL{L})$, we clearly have a element $[\varpi] \in H^{+}(\wedge W)^{4s+i-1}$ such that $\xi([\varpi]) = \varpi$.

Denote by $\lbrace -,- \rbrace$ the bilinear pairing between cohomology and homology defined by $\lbrace [f],[c] \rbrace := f(c)$, since we work on $\mathbf{Q}$ that pairing is also non degenerate. By the definition of the pairing $\langle -,- \rangle$ (see \cite[p 172-173]{Felix2001}) we have
\[
\langle \xi([\varpi]), [\varphi] \rangle = \lbrace H(m_{\cotrL{L}})[\varpi],\textsc{Hur}_{4s+i-1}(\varphi)\rbrace \neq 0.
\]

Since $H(m_{\cotrL{L}})[\varpi] = [L]^{\ast}$, we have $\textsc{Hur}_{4s+i-1}(\varphi)=q[L]$ with $q \in \mathbf{Q}-\lbrace 0 \rbrace$. Then 
\[
\phi:=\frac{1}{q}\varphi \in \pi_{4s+i-1}(\cotrL{L})\otimes \mathbf{Q}
\]
is the map we wanted.
\end{proof}

\begin{lem}\label{lem:indecomposable}
$\varpi$ is an indecomposable element of $(\Min{\cotrL{L}},d)$, that is 
\[
\varpi \in W^{4s+i-1}
\]
where $W = \oplus_{k \geq 0} W^{k}$ is the graded vector space generating $\Min{\cotrL{L}}$.
\end{lem}

\begin{proof}
The elements of degree $2s+\frac{i-1}{2}$ of $(\Min{\cotrL{L}},d)$ come from the Lagrangian $V$ so for all elements $x,y \in  \Min{\cotrL{L}}^{2n+\frac{i-1}{2}}$ there exists an element $z \in \Min{\cotrL{L}}^{4n+i-2}$ such that $dz = x \cdot y$, in particular, none of these products are equal to $\varpi$. For degree reasons these were the only elements we had to care about.
\end{proof}

Denote by $\phi$ the element of $\pi_{4s+i-1}(\cotrL{L})\otimes \mathbf{Q}$ obtained by this process, like in the general case, and consider the homotopy pushout.
\[
\begin{tikzpicture}[injection/.style={right hook->,fill=white, inner sep=2pt}]

\matrix (m)[matrix of math nodes, row sep=3em, column sep=5em, text height=1.5ex, text depth=0.25ex]
{ S^{4s+i-1}                &     D^{4s+i}          \\
\cotrL{L}                   &     \cotr{\overline{\mathcal{L}},\phi}{L}   \\};

\path[injection]
(m-1-1) edge node[auto] {} (m-1-2);
\path[->]
(m-2-1) edge node[auto] {} (m-2-2);
\path[->]
(m-1-1) edge node[auto,swap] {$\phi$} (m-2-1);
\path[->]
(m-1-2) edge node[left=2.5em] {HPO} (m-2-2);
\end{tikzpicture}
\]

\begin{propo}
$(\cotr{\overline{\mathcal{L}},\phi}{L},L)$ is a Poincaré duality pair if and only if $\phi_{4s+i-1}$ is an isomorphism.
\end{propo}

We denote by $\DP{X}$ the space obtained by the glueing of the two Poincaré duality pairs $(X_{reg},\partial X_{reg} ) = (X_{reg}, L)$ and $(\cotrL{L}_{\phi},L)$ following the theorem \ref{thm:glueing}. We have the last part of the theorem \ref{thm:main_thm_one_sing}.

\begin{propo}
\label{propo:2k+1_L_space}
Let $X$ be a compact, connected oriented normal pseudomanifold of dimension $n=4s+1$ or $n=4s+3$ with one isolated singularity $\sigma$ of simply connected link $L$. Suppose moreover that $X$ is an L-space. Then there exists a rational Poincaré approximation $\DP{X}$ of $X$.
\end{propo}

Just like in the even dimensional case with the proposition \ref{propo:equiv_IX_IXphi}. We can relate the spaces $\DP{X}$ to the intersection and Lagrangian intersections spaces, and get more precision on how to attach the top cell in the simply connected case. When $X$ is a Witt space, we denote by $IX = \I{m}{X} = \I{n}{X}$.

\begin{propo}
\label{propo:equiv_IX_IXphi_2s+1}
Let $X$ be a compact, connected oriented normal pseudomanifold of dimension $n=2s+1$ with one isolated singularity of link $L$ simply connected. 
\begin{enumerate}
\item Suppose $X$ is a Witt space. The space $\DP{X}$ is then rationally homotopy equivalent to $IX \cup e^{2s+1}$. If moreover $H^{1}(X_{reg})=0$, then $e^{2s+1}$ is attached by ordinary Whitehead products (not iterated) with respect to some basis of $\pi_{\ast}(IX) \otimes \mathbf{Q}$ and the rational homotopy type of $\DP{X}$ is determined by $IX$.
\item Suppose $X$ is an L-space. The space $\DP{X}$ is then rationally homotopy equivalent to $\nI{\mathcal{L}}{X} \cup e^{2s+1}$. If moreover $H^{1}(X_{reg})=0$, then $e^{2s+1}$ is attached by ordinary Whitehead products (not iterated) with respect to some basis of $\pi_{\ast}(\nI{\mathcal{L}}{X}) \otimes \mathbf{Q}$ and the rational homotopy type of $\DP{X}$ is determined by $\nI{\mathcal{L}}{X}$.
\end{enumerate}
\end{propo}

\begin{proof}
The proof is the same as for the proposition \ref{propo:equiv_IX_IXphi} unless we consider the following diagram when $X$ is a Witt space.
\[
\begin{tikzpicture}[injection/.style={right hook->,fill=white, inner sep=2pt}]

\matrix (m)[matrix of math nodes, row sep=2em, column sep=5em, text height=1.5ex, text depth=0.25ex]
{             & L=\partial X_{reg} & X_{reg}         \\
S^{2s}        & \cotr{}{L}         & IX        \\
D^{2s+1}      & \cotr{\phi}{L}     & \DP{X} \\};

\path[->]
(m-1-2) edge node[style] {} (m-2-2);
\path[->]
(m-1-3) edge node[left=2em] {(I) HPO} (m-2-3);
\path[->]
(m-2-1) edge node[auto,swap] {$i_{2}$} (m-3-1);
\path[->]
(m-2-2) edge node[left=2em] {(III) HPO} (m-3-2);
\path[->]
(m-2-3) edge node[left=2.5em] {(II)} (m-3-3);

\path[injection]
(m-1-2) edge node[auto] {$i_{1}$} (m-1-3);
\path[->]
(m-2-1) edge node[auto]{$\phi$} (m-2-2);
\path[->]
(m-2-2) edge node[style] {} (m-2-3);
\path[->]
(m-3-1) edge node[style] {} (m-3-2);
\path[->]
(m-3-2) edge node[style] {} (m-3-3);
\end{tikzpicture}
\]

and the following diagram when $X$ is an L-space.
\[
\begin{tikzpicture}[injection/.style={right hook->,fill=white, inner sep=2pt}]

\matrix (m)[matrix of math nodes, row sep=2em, column sep=5em, text height=1.5ex, text depth=0.25ex]
{             & L=\partial X_{reg}                        & X_{reg}         \\
S^{2s}        & \cotr{\overline{\mathcal{L}}}{L}          & \nI{\mathcal{L}}{X}        \\
D^{2s+1}      & \cotr{\overline{\mathcal{L}},\phi}{L}     & \DP{X} \\};

\path[->]
(m-1-2) edge node[style] {} (m-2-2);
\path[->]
(m-1-3) edge node[left=2em] {(I) HPO} (m-2-3);
\path[->]
(m-2-1) edge node[auto,swap] {$i_{2}$} (m-3-1);
\path[->]
(m-2-2) edge node[left=2em] {(III) HPO} (m-3-2);
\path[->]
(m-2-3) edge node[left=2.5em] {(II)} (m-3-3);

\path[injection]
(m-1-2) edge node[auto] {$i_{1}$} (m-1-3);
\path[->]
(m-2-1) edge node[auto]{$\phi$} (m-2-2);
\path[->]
(m-2-2) edge node[style] {} (m-2-3);
\path[->]
(m-3-1) edge node[style] {} (m-3-2);
\path[->]
(m-3-2) edge node[style] {} (m-3-3);
\end{tikzpicture}
\]
\end{proof}

The considerations whether the approximations are good or very good come from the study of the rational homology of these diagrams of homotopy pushouts
\[
\begin{tikzpicture}[injection/.style={right hook->,fill=white, inner sep=2pt}]

\matrix (m)[matrix of math nodes, row sep=3em, column sep=5em, text height=1.5ex, text depth=0.25ex]
{ L                &     X_{reg}          \\
\cotr{\phi}{L}     &     \DP{X}     \\
   \ast            &     X          \\};
  
\path[injection]
(m-1-1) edge node[auto] {} (m-1-2);
\path[->]
(m-2-1) edge node[auto] {} (m-2-2);
\path[->]
(m-3-1) edge node[auto] {} (m-3-2);

\path[->]
(m-1-1) edge node[auto] {} (m-2-1);
\path[->]
(m-2-1) edge node[auto] {} (m-3-1);

\path[->]
(m-1-2) edge node[auto,swap] {$\phi$} (m-2-2);
\path[->]
(m-2-2) edge node[auto,swap] {$\psi$} (m-3-2);
\path[injection,bend left]
(m-1-2) edge node[auto] {$i$} (m-3-2);
\end{tikzpicture}
\quad
\begin{tikzpicture}[injection/.style={right hook->,fill=white, inner sep=2pt}]

\matrix (m)[matrix of math nodes, row sep=3em, column sep=5em, text height=1.5ex, text depth=0.25ex]
{ L                                      &     X_{reg}          \\
\cotr{\overline{\mathcal{L}},\phi}{L}    &     \DP{X}     \\
   \ast                                  &     X          \\};
  
\path[injection]
(m-1-1) edge node[auto] {} (m-1-2);
\path[->]
(m-2-1) edge node[auto] {} (m-2-2);
\path[->]
(m-3-1) edge node[auto] {} (m-3-2);

\path[->]
(m-1-1) edge node[auto] {} (m-2-1);
\path[->]
(m-2-1) edge node[auto] {} (m-3-1);

\path[->]
(m-1-2) edge node[auto,swap] {$\phi$} (m-2-2);
\path[->]
(m-2-2) edge node[auto,swap] {$\psi$} (m-3-2);
\path[injection,bend left]
(m-1-2) edge node[auto] {$i$} (m-3-2);
\end{tikzpicture}
\]
for the even dimensional case and the Witt space case, or for the L-space case.

\subsection{The multiple isolated singularities case}
\label{subsubsec:multi_sing}

The theorem \ref{thm:glueing} did not make any assumptions on the connectivity of the pairs $(X_{j}, Y_{j}, [x_{j}])$, so in fact we can apply everything that was above to the case of a pseudomanifold with more than one isolated singularity.
\begin{thm}[Multiple isolated singularities case]
Let $X$ be a compact, connected oriented normal pseudomanifold of dimension $n$ with only isolated singularities $\Sigma = \lbrace \sigma_{1}, \dots, \sigma_{\nu}\rbrace$, $\nu >1$, of links $L_{i}$ simply connected. Then, 
\begin{enumerate}
\item If $n = 2s$, there exists a good rational Poincaré approximation $\DP{X}$ of $X$. Moreover, if $\dim X \equiv 0 \mod 4$, the Witt class associated to the intersection form $b_{\DP{X}}$ is the same as the Witt class associated to the middle intersection cohomology of $X$ in $W(\mathbf{Q})$.
\item If $n= 2s+1$ and $X$ is either a Witt space or an L-space there exists a good rational Poincaré approximation $\DP{X}$ of $X$. Moreover is $X$ is Witt space $\DP{X}$ is a very good rational Poincaré approximation of $X$
\end{enumerate}
\end{thm}

Just like before, the considerations whether the approximations are good or very good come from the study of the rational homology of these diagrams of homotopy pushouts
\[
\begin{tikzpicture}[injection/.style={right hook->,fill=white, inner sep=2pt}]

\matrix (m)[matrix of math nodes, row sep=3em, column sep=5em, text height=1.5ex, text depth=0.25ex]
{ \bigsqcup_{\sigma_{i}} L_{i}                     &     X_{reg}          \\
  \bigsqcup_{\sigma_{i}} \cotr{\phi_{i}}{L_{i}}    &     \DP{X}     \\
  \bigsqcup_{\sigma_{i}} \ast                      &     X          \\};
  
\path[injection]
(m-1-1) edge node[auto] {} (m-1-2);
\path[->]
(m-2-1) edge node[auto] {} (m-2-2);
\path[->]
(m-3-1) edge node[auto] {} (m-3-2);

\path[->]
(m-1-1) edge node[auto] {} (m-2-1);
\path[->]
(m-2-1) edge node[auto] {} (m-3-1);

\path[->]
(m-1-2) edge node[auto,swap] {$\phi$} (m-2-2);
\path[->]
(m-2-2) edge node[auto,swap] {$\psi$} (m-3-2);
\path[injection,bend left]
(m-1-2) edge node[auto] {$i$} (m-3-2);
\end{tikzpicture}
\quad
\begin{tikzpicture}[injection/.style={right hook->,fill=white, inner sep=2pt}]

\matrix (m)[matrix of math nodes, row sep=3em, column sep=5em, text height=1.5ex, text depth=0.25ex]
{ \bigsqcup_{\sigma_{i}} L_{i}                                      &     X_{reg}         \\
  \bigsqcup_{\sigma_{i}} \cotr{\overline{\mathcal{L}},\phi_{i}}{L}  &     \DP{X}     \\
  \bigsqcup_{\sigma} \ast                                           &     X          \\};
  
\path[injection]
(m-1-1) edge node[auto] {} (m-1-2);
\path[->]
(m-2-1) edge node[auto] {} (m-2-2);
\path[->]
(m-3-1) edge node[auto] {} (m-3-2);

\path[->]
(m-1-1) edge node[auto] {} (m-2-1);
\path[->]
(m-2-1) edge node[auto] {} (m-3-1);

\path[->]
(m-1-2) edge node[auto,swap] {$\phi$} (m-2-2);
\path[->]
(m-2-2) edge node[auto,swap] {$\psi$} (m-3-2);
\path[injection,bend left]
(m-1-2) edge node[auto] {$i$} (m-3-2);
\end{tikzpicture}
\]
for the even dimensional case and the Witt space case, or for the L-space case.

\subsubsection{The even dimensional case}

Let $X$ be a compact oriented normal pseudomanifold of dimension $n=2s$ with only isolated singularities $\Sigma= \lbrace \sigma_{1}, \dots, \sigma_{\nu} ; \nu >1\rbrace$ of simply connected links $L_{i}$.

The rational Hurewicz theorem gives us maps $\phi_{i}$ such that the pairs $(\cotr{\phi_{i}}{L_{i}}, L_{i})$ are Poincaré duality pairs for all $i$. Denote by $[e_{\phi}]_{i}$ the induced orientation class in $H_{2s}(\cotr{\phi_{i}}{L_{i}}, L_{i})$. 

The pair $(X_{reg}, \partial X_{reg})$ with $\partial X_{reg} = \sqcup_{\sigma_{i}} L_{i}$ is still a manifold with boundary, thus satisfies Poincaré-Lefschetz duality and is a Poincaré duality pair. The theorem \ref{thm:glueing} then applies and, with the same notation, $\DP{X}$ is an oriented Poincaré complex of dimension $2s$ without boundary and of orientation class given by 
\[
[\DP{X}] = i^{-1}(i_{1} \oplus i_{2}([X_{reg},\partial X_{reg}], [-e_{\phi}]_{1}, \dots, [-e_{\phi}]_{r})),
\]
and all the results obtained before remain true except for the proposition \ref{propo:equiv_IX_IXphi} which has to be modified. We have to take the normal intersection space $\mathcal{I}X = \nI{m}{X} = \nI{n}{X}$ to modify the proposition \ref{propo:equiv_IX_IXphi}. Which becomes

\begin{propo}
Let $X$ be a compact, connected oriented normal pseudomanifold of dimension $n=2s$ with only isolated singularities $\Sigma = \lbrace \sigma_{1}, \dots, \sigma_{\nu} ; \nu >1 \rbrace$ of links $L_{i}$ simply connected. Suppose moreover that $H^{1}(X_{reg})=0$. Then $\DP{X}$ is rationally homotopy equivalent to $\tr{2s-1}{\mathcal{I}X} \cup e^{2s}$ where  $\tr{2s-1}{\mathcal{I}X}$ is the $(2s-1)$-truncation of $\mathcal{I}X$ and where $e^{2s}$ is attached by ordinary Whitehead products (not iterated) with respect to some basis of $\pi_{\ast}(\tr{2s-1}{\mathcal{I}X}) \otimes \mathbf{Q}$ and the rational homotopy type of $\DP{X}$ is determined by $\tr{2s-1}{\mathcal{I}X}$.
\end{propo}

\begin{proof}
Consider the following diagram, obtained by the construction of $\DP{X}$

\[
\begin{tikzpicture}[injection/.style={right hook->,fill=white, inner sep=2pt}]

\matrix (m)[matrix of math nodes, row sep=2em, column sep=5em, text height=1.5ex, text depth=0.25ex]
{                                      & \bigsqcup_{\sigma_{i}} L_{i}=\partial X_{reg}   & X_{reg}         \\
\bigsqcup_{\sigma_{i}} S_{i}^{2s-1}    & \bigsqcup_{\sigma_{i}} \cotr{}{L_{i}}           & \mathcal{I}X        \\
\bigsqcup_{\sigma_{i}} D_{i}^{2s}      & \bigsqcup_{\sigma_{i}} \cotr{\phi_{i}}{L_{i}}   & \DP{X} \\};

\path[->]
(m-1-2) edge node[style] {} (m-2-2);
\path[->]
(m-1-3) edge node[left=2em] {(I) HPO} (m-2-3);
\path[->]
(m-2-1) edge node[auto,swap] {$i_{2}$} (m-3-1);
\path[->]
(m-2-2) edge node[left=2em] {(III) HPO} (m-3-2);
\path[->]
(m-2-3) edge node[left=2.5em] {(II)} (m-3-3);

\path[injection]
(m-1-2) edge node[auto] {$i_{1}$} (m-1-3);
\path[->]
(m-2-1) edge node[auto]{$\bigsqcup_{\sigma_{i}} \phi_{i}$} (m-2-2);
\path[->]
(m-2-2) edge node[style] {} (m-2-3);
\path[->]
(m-3-1) edge node[style] {} (m-3-2);
\path[->]
(m-3-2) edge node[style] {} (m-3-3);
\end{tikzpicture}
\]

With the same arguments than for the unique isolated singularity case, $\DP{X}$ is rationally homotopy equivalent to
\[
\DP{X} \simeq \mathcal{I}X \cup (\bigcup_{\phi_{i}} e^{2s}_{i}).
\]

Now, $H^{1}(X_{reg})=0$ so $\DP{X}$ is simply connected and the theorem \ref{thm:Q_cell_model} gives a rational homotopy equivalence
\[
\varphi \col \widetilde{\DP{X}} \longrightarrow \DP{X}
\]
such that the differential in the integral cellular chain complex of $\widetilde{\DP{X}}$ is identically zero. This implies that there is only one top dimensional cell on $\widetilde{\DP{X}}$, we have a attaching map $\theta$ and a cell $e^{2s}$ such that
\[
\widetilde{\DP{X}} = X_{0} \cup_{\theta} e^{2s}.
\]
Let us now determine $X_{0}$. By the theorem \ref{thm:Q_cell_model} $X_{0}$ is a CW-complex of dimension $2s-1$ and by the Poincaré duality of $\DP{X}$ we have 
\[
H_{2s-1}(X_{0}) = H_{2s-1}(\DP{X}) \cong H^{1}(\DP{X}) = H^{1}(X_{reg})=0.
\]
For any other $r \leq 2s-2$ we have by construction $H_{r}(X_{0}) = H_{r}(\mathcal{I}X)$.

\[
\begin{tikzpicture}
\matrix (m)[matrix of math nodes, row sep=3em, column sep=2em, text height=2ex, text depth=0.25ex]
{ X_{0}                  &                    &                     &              & X_{0} \cup_{\theta} e^{2s}   \\
 \tr{2s-1}{\mathcal{I}X} &  \mathcal{I}X/2s-1 & \mathcal{I}X^{2s-1} & \mathcal{I}X & \DP{X}\\};

\path[->]
(m-1-1) edge node[auto] {$\mathrm{incl}$} (m-1-5);
\path[->]
(m-2-1) edge node[auto,swap] {$i$} (m-2-2);
\path[->]
(m-2-2) edge node[auto,swap] {$h$} (m-2-3);
\path[->]
(m-2-3) edge node[auto,swap] {$j$} (m-2-4);
\path[->]
(m-2-4) edge node[auto,swap] {} (m-2-5);
\path[->]
(m-1-1) edge node[auto,swap] {$\varphi_{|}$} (m-2-1);
\path[->]
(m-1-5) edge node[auto] {$\varphi$} (m-2-5);
\end{tikzpicture}
\]

The above diagram, where the maps $i$ and $j$ are cellular inclusions and $h$ a cellular homotopy equivalence as in the definition of the homology truncation \ref{def:ntroncation}, commutes and by the same argument that the one given before the proposition \ref{propo:truncation}, $\varphi$ restricts to a rational homotopy equivalence
\[
\varphi_{|} \col X_{0} \longrightarrow \tr{2s-1}{\mathcal{I}X}.
\]

Then, up to rational homotopy equivalence, we have
\[
\DP{X} = \tr{2s-1}{\mathcal{I}X} \cup_{\theta} e^{2s}.
\]

The theorem \ref{thm:stasheff} then tells us that $e^{2s}$ is attached by ordinary Whitehead products (not iterated) with respect to some basis of $\pi_{\ast}(\tr{2s-1}{\mathcal{I}X}) \otimes \mathbf{Q}$ and the rational homotopy type of $\DP{X}$ is determined by $\tr{2s-1}{\mathcal{I}X}$.
\end{proof}

In particular, if we have only one normal isolated singularity and if $\DP{X}$ is simply connected, then $H_{2s-1}(\DP{X}) = H_{2s-1}(\mathcal{I}X)= H_{2s-1}(IX)=0$ and $\tr{2s-1}{\mathcal{I}X} \simeq IX$. We then get back the proposition \ref{propo:equiv_IX_IXphi}.
 
The others results remain true, in particular $\DP{X}$ is a rational Poincaré duality space and we have the first part of the theorem \ref{thm:main_thm_mult_sing}.

\begin{propo}
If $\dim X = 2s$, then $\DP{X}$ is a good rational Poincaré approximation of $X$. Moreover, if $\dim X \equiv 0 \mod 4$, then the Witt class associated to the intersection form $b_{\DP{X}}$ is the same that the Witt class associated to the middle intersection cohomology of $X$ in $W(\mathbf{Q})$.
\end{propo}

\subsubsection{The odd dimensional case}
\begin{propo}
Let $X$ be a compact, connected oriented normal pseudomanifold of dimension $n=2s+1$ with only isolated singularities $\Sigma= \lbrace \sigma_{1}, \dots, \sigma_{\nu} ; \nu >1\rbrace$ of links $L_{i}$ simply connected. 
\begin{enumerate}
\item Suppose moreover $X$ is a Witt space, then $\DP{X}$ exists and is a very good rational Poincaré duality space,
\item Suppose moreover $X$ is an L-space, then $\DP{X}$ is a good rational Poincaré approximation of $X$.
\end{enumerate}
\end{propo}

\begin{proof}
The proof is the same as the ones for the propositions \ref{propo:2k+1_Witt_case} and \ref{propo:2k+1_L_space} but with multiple links $L_{i}$.
\end{proof}

\section{Examples and applications}
\subsection{Real Algebraic varieties}

If $V$ is a real algebraic variety of even dimension with isolated singularities $\Sigma = \lbrace \sigma_{1}, \dots, \sigma_{\nu} \rbrace$ and an oriented regular part $V_{reg}$. We can apply the homological truncation and then by the use of the precedents results construct a rational Poincaré approximation $\DP{V}$.

The odd dimensional is more interesting. Suppose that $V$ is a real algebraic variety of  with multiple isolated singularities of odd dimension. If the regular part $V_{reg}$ of $V$ is oriented then $V$ is automatically an L-space due to the following result of Selman Akbulut and Henry King : 

\begin{thm}[\cite{Akbulut1981}]
Let $V$ be a compact topological space. Then the following are equivalent :
\begin{enumerate}
\item $V$ is homeomorphic to a real algebraic set with isolated singularities.
\item $V$ is homeomorphic to the quotient obtained by taking a smooth closed manifold $M$ and collapsing each $L_{i}$ to point a point where $L_{i}$, $i = 1, \dots, \nu$ is a collection of disjoint smooth subpolyhedra of $M$.
\item $V = M \cup \bigcup_{i=1}^{\nu} cL_{i}$ where $M$ and $L_{i}$ are smooth compact manifolds, $\partial M$ is the disjoint union of the $L_{i}$'s, each $L_{i}$ bounds a smooth compact manifolds and $L_{i} \times 1 \subset cL_{i}$ is identified with $L_{i} \subset M$.
\end{enumerate}
\end{thm}

Consider then $V$ a oriented real algebraic variety of dimension $n=4s+1$ with $\nu$ isolated singularities. Then by the third equivalence we have
\[
V = M \cup \bigcup_{i=1}^{\nu} cL_{i}
\]
and each link $L_{i}$ is a smooth compact manifold of dimension $4s$ and is the boundary of a $4s+1$ smooth compact manifold. Then by the lemma \ref{lem:thom_hyperb_link} we have that
\[
[b_{i}] = 0 \in W(\mathbf{Q}) \, \forall i.
\]

We can perform a Lagrangian truncation and we have our rational Poincaré approximation $\DP{V}$.

Note that if $V$ is of dimension $4s+3$ then we don't need this result because the bilinear form $b_{L}$ would be skew-symmetric.

We have the following result
\begin{propo}
Every oriented real algebraic variety $V$ with only isolated singularities and simply connected links admits at least a good rational Poincaré approximation $\DP{V}$.
\end{propo}

\subsection{Hypersurfaces with nodal singularities}

Let $V$ be a complex projective hypersurface with one nodal singularity such that $\dim_{\mathbf{C}} V = 3$. The link of this singularity is $L = S^{2} \times S^{3}$, by theorem \ref{thm:stasheff} the link of the singularity is rationally homotopy equivalent to 
\[
L \simeq (S^{2} \vee S^{3}) \bigcup_{\psi} e^{5}
\]
where $\psi$ is a Whitehead product.

Since $\overline{m}(6) = \overline{n}(6) = 2$, the homological truncation of the link is
\[
\tr{2}{L} = S^{2}
\]
and the cotruncation is rationally homotopy equivalent to
\[
\cotr{2}{L} = S^{3} \vee S^{5}.
\]

To see this, just compute the cohomology algebra of the cotruncation. The rational Hurewicz theorem \ref{thm:Qhurewicz} then says that we have the isomorphism
\[
\pi_{5}(\cotr{2}{L}) \otimes \mathbf{Q} \overset{\cong}{\longrightarrow} H_{5}(\cotr{2}{L}) \cong H_{5}(S^{5}).
\]
The cell attachment $\phi$ obtained by this isomorphism then kill the $5$-sphere of the cotruncation. That is we have
\[
\cotr{\phi}{L} = (S^{3} \vee S^{5})\cup_{s_{5}^{\sharp}}e^{6} \simeq S^{3}.
\]

But $\cotr{\phi}{L} = S^{3} \simeq D^{3}\times S^{3}$ ans since $\partial (D^{3}\times S^{3}) = S^{2} \times S^{3}$, the pair $(\cotr{\phi}{L}, L) = (D^{3}\times S^{3}, S^{2} \times S^{3})$ is a Poincaré duality pair. The space $\DP{V}$ is a good rational Poincaré approximation of $X$.

This construction extends to multiple isolated singularities and higher dimension complex hypersurfaces with nodal singularities.

\subsection{Thom Spaces}

\begin{defi}
Let $B$ be a compact, connected, oriented manifold of dimension $m$ and $E$ a fiber bundle over $B$ of rank $m'$,
\[
\mathbf{R}^{m'} \longrightarrow E \longrightarrow B.
\]
The Thom space $Th(E)$ of the fiber bundle $E$ is defined as the homotopy cofiber of the map
\[
S_{E} \longrightarrow D_{E}
\]
\noindent where $S_{E}$ and $D_{E}$ are respectively the sphere bundle and disk bundle associated to $E$.
\end{defi}

$Th(E)$ is then a pseudomanifold of dimension $m+m'$, the singularity is the compactification point, its link is the sphere bundle $S_{E}$ and the regular part of $Th(E)$ is the disk bundle $D_{E}$.

We show that in the case of an odd dimensional Thom space $Th(E)$ is either an L-space or a Witt space whether the rank of the vector bundle is lesser than the dimension of the base space or not.

\begin{thm}
\label{thm:Thom_L_Witt}
Suppose $m'>0$.
\begin{enumerate}
\item Let $\mathbf{R}^{2m'} \longrightarrow E \longrightarrow B^{2m+1}$ with $B$ be a manifold of dimension $2m+1$ and $E$ a fiber bundle over $B$ of rank $2m'$. Then,
\begin{itemize}
\item if $m' \leq m+1$, $Th(E)$ is an L-space,
\item if $m' > m+1$, $Th(E)$ is a Witt space if and only if $H^{m+m'}(B)=0$.
\end{itemize}

\item Let $\mathbf{R}^{2m'+1} \longrightarrow E \longrightarrow B^{2m}$ with $B$ be a manifold of dimension $2m$ and $E$ a fiber bundle over $B$ of rank $2m'+1$. Then,
\begin{itemize}
\item if $m' \leq m$, $Th(E)$ is an L-space,
\item if $m' > m$, $Th(E)$ is a Witt space if and only if $H^{m+m'}(B)=0$.
\end{itemize}
\end{enumerate}
\end{thm} 

\begin{proof}
Consider $\mathbf{R}^{2m'} \longrightarrow E \longrightarrow B^{2m+1}$. 

In order to know if $Th(E)$ is an L-space or a Witt space we have to look at $H^{m+m'}(S_{E})$ where $S_{E}$ is the sphere bundle associated to the vector bundle $E$ (see definitions \ref{def:Wittspace} and \ref{def:Lspace}). To compute $H^{m+m'}(S_{E})$ we use the cohomological Leray-Serre spectral sequence associated to the fiber bundle
\[
S^{2m'-1} \longrightarrow S_{E} \longrightarrow B^{2m+1}.
\]

We have $E_{2}^{p,q} = H^{p}(B;H^{q}(S^{2m'-1})) = 0$ if $q \neq 0, 2m'-1$ and 
\[
d_{2m'} : E_{2m'}^{p,2m'-1} \longrightarrow E_{2m'}^{p+2m',0}
\]
is the only non-zero differential which is defined by 
\[
d_{2m'} : E_{2m'}^{0,2m'-1} \longrightarrow E_{2m'}^{2m',0}
\]
with $d_{2m'}(a) = eu(E) \in H^{2m'}(B)$ where $a$ is the generator of $H^{2m'-1}(S^{2m'-1})$ and $eu(E) \in H^{2m'}(B)$ the Euler class of the sphere bundle. Since this is the only non-zero differential we have 
\[
H^{m+m'}(S_{E}) = E_{2m'+1}^{m+m',0} \oplus E_{2m'+1}^{m-m'+1,2m'-1}.
\]

Suppose that $m' \leq m+1$. 

The summand $E_{2m'+1}^{m-m'+1,2m'-1}$ is well defined and using the structure product of the spectral sequence, we see that the product of two elements belonging to the same summand of $H^{m+m'}(S_{E})$ is zero. Then by Poincaré duality the symmetric bilinear form 
\[
E_{2m'+1}^{m+m',0} \times E_{2m'+1}^{m-m'+1,2m'-1} \longrightarrow E_{2m'+1}^{2m+1,2m'-1} \cong \mathbf{Q}\omega a
\] 
induced by the product and where $\omega \in H^{2m+1}(B)$ is the fundamental class of the manifold $B$ is non degenerate. Thus, provided than one of the two summand is non zero, the symmetric bilinear form is then hyperbolic and $S_{E}$ is an L-space.

Suppose that $m' > m+1$.

Then $E_{2m'+1}^{m-m'+1,2m'-1}=0$ and $E_{2m'+1}^{m+m',0} = H^{m+m'}(B)$ and $S_{E}$ is a Witt space if and only is $H^{m+m'}(B)=0$.

We now consider $\mathbf{R}^{2m'+1} \longrightarrow E \longrightarrow B^{2m}$. By the same arguments we have
\[
H^{m+m'}(S_{E}) = E_{2m'+2}^{m+m',0} \oplus E_{2m'+2}^{m-m',2m'}.
\]

If $m' \leq m$ then $E_{2m'+2}^{m-m',2m'}$ is well defined, the same arguments about the product structure and Poincaré duality imply that $S_{E}$ is an L-space.

If $m' > m$ then $E_{2m'+2}^{m-m',2m'}=0$, $E_{2m'+2}^{m+m',0}=H^{m+m'}(B)$ and $Th(E)$ is a Witt space if and only if $H^{m+m'}(B)=0$.
\end{proof}

\begin{cor}
For any complex line bundle $\mathbf{C} \longrightarrow E \longrightarrow B^{k}$ with $k \geq 2$, the Thom space $Th(E)$ is an L-space.
\end{cor}

Let $B^{9} := (S^{3} \times S^{6}) \sharp (S^{4} \times S^{5})$, where $\sharp$ denotes the connected sum, and let $f : B^{9} \rightarrow S^{4}$ the composition of the following contraction map $q$ and projection map $p$ :
\[
B^{9} \overset{q}{\longrightarrow} S^{4}\times S^{5} \overset{p}{\longrightarrow} S^{4}.
\]
Let $E$ be the fiber bundle over $B^{9}$ that is the pullback along $f$ of the tangent space over $S^{4}$,
\[
\begin{tikzpicture}

\matrix (m)[matrix of math nodes, row sep=3em, column sep=5em, text height=1.5ex, text depth=0.25ex]
{ \mathbf{R}^{4}           &    \mathbf{R}^{4}  \\
   E := f^{\ast}(TS^{4})   &     TS^{4}         \\
        B^{9}              &     S^{4}          \\};

\path[->]
(m-1-1) edge node[auto] {} (m-2-1);
\path[->]
(m-2-1) edge node[auto] {} (m-3-1);

\path[->]
(m-3-1) edge node[auto] {$f$} (m-3-2);
\path[->]
(m-2-1) edge node[auto] {} (m-2-2);

\path[->]
(m-1-2) edge node[auto] {} (m-2-2);
\path[->]
(m-2-2) edge node[auto] {} (m-3-2);
\end{tikzpicture}
\]

The Thom space $Th(E)$ associated to this bundle is a pseudomanifold of dimension 13 and is an $L$-space by theorem \ref{thm:Thom_L_Witt}.

Using the Leray-Serre spectral sequence we have
\[
H^{6}(S_{E}) = \mathbf{Q}s_{6} \oplus \mathbf{Q}s_{3}a.
\]
where $s_{i}$ represents the generator of the sphere $S^{i}$ in $B^{9}$ and $a$ the generator of the fiber $S^{3}$ of the sphere bundle. Denote by $\omega$ the fundamental class of $B^{9}$ with $\omega = s_{4}s_{5} =s_{3}s_{6}$.

Using the product structures of the spheres $S^{6}$ and $S^{3}$, we have $s_{6}s_{6} =0$ and $(s_{3}a)(s_{3}a)=0$, but since $\omega = s_{3}s_{6}$ the matrix of the intersection form
\[
H^{6}(S_{E}) \times H^{6}(S_{E}) \longrightarrow \mathbf{Q}
\]
in the base $(s_{6}, s_{3}a)$ is given by
\[
\begin{pmatrix}
0 & 1 \\
1 & 0
\end{pmatrix}.
\]
The intersection form is then hyperbolic and both factors $\mathbf{Q}s_{6}$ and $\mathbf{Q}s_{3}a$ are Lagrangian subspaces.

We now construct a rational model of $\DP{Th(E)}$. For that we'll need a surjective model of $S_{E} \hookrightarrow D_{E}$ and a model of the Lagrangian truncation $\cotrL{S_{E}}$. A surjective model of $S_{E} \hookrightarrow D_{E}$ is given by 
\[
A(D_{E}) \overset{\varphi}{\twoheadrightarrow} A(S_{E})
\]
with 
\[
\begin{cases}
A(S_{E})   & = (A(B) \otimes \wedge a, d) \text{ with } da = s_{4}\\
A(D_{E})   & = (A(B) \otimes \wedge (a,b), D) \text{ with } Da = s_{4} - b \\
\varphi_{|A(B) \otimes \wedge a} & = \mathrm{id} \\
\varphi(b) & = 0 
\end{cases}
\]
where $A(B)$ is a rational model of the base space $(S^{3} \times S^{6}) \sharp (S^{4} \times S^{5})$ which is given by
\[
A(B) =(\wedge (s_{3},s_{4},s_{5},s_{6},\beta_{6},\beta_{7,1},\beta_{7,2},\beta_{8},\dots),d)
\]
with $|s_{i}| = |\beta_{i}| = i$ and
\[
\begin{cases}
d \beta_{6}     & = s_{3}s_{4}\\
d \beta_{7,1} & = s_{4}^{2} \\
d \beta_{7,2} & = s_{5}s_{3} \\
d \beta_{8}     & = s_{3}s_{6}-s_{4}s_{5}.
\end{cases}
\]

Since $\dim B = 9$ we only gave elements of the model of $B$ up to degree 9, the rest of the model being an acyclic part. That is for every element $\alpha_{k}$ of degree $k \geq 10$ such that $d \alpha_{k} =0$, there is an element $\beta_{k-1}$ such that $d \beta_{k-1}= \alpha_{k}$. In fact we can take take a better model for $A(S_{E})$ and $A(D_{E})$ because the base space $B$ is a formal space, that is we have a quasi isomorphism
\[
\psi : (A(B),d) \longrightarrow (H(B),0)
\]
given by
\[
\begin{cases}
\psi(s_{i})          & = s_{i} \\
\psi(\beta_{i})      & = 0 \\
\psi(A^{\geq 10}(B)) & = 0.
\end{cases}
\]

The models we use are then
\[
\begin{cases}
A(S_{E})   & = (H(B) \otimes \wedge a, d) \text{ with } da = s_{4}\\
A(D_{E})   & = (H(B) \otimes \wedge (a,b), D) \text{ with } Da = s_{4} - b \\
\varphi_{|H(B) \otimes \wedge a} & = \mathrm{id} \\
\varphi(b) & = 0 
\end{cases}
\]

By adapting the proposition \ref{propo:truncation} and then using the lemma \ref{lem:cotruncation} we can show that a model of the Lagrangian cotruncation is given by
\[
(A(\cotrL{S_{E}}),d) = (\mathbf{Q} \oplus I_{\mathcal{L}},d)
\]
where $I_{\mathcal{L}}$ is the differential ideal given in this case by a choice of generator for one of the Lagrangian subspaces, a complementary of $\ker (d : A^{6}(S_{E}) \rightarrow A^{7}(S_{E}))$ and all the cochains of degree greater or equal to 7 of $A(S_{E})$. A choice of $A(\cotrL{S_{E}}))$ is given by
\[
\mathbf{Q} \oplus (\mathbf{Q}s_{6}\otimes (H(B)\otimes \wedge a)^{\geq 7})  
\]
and the model of $\I{\mathcal{L}}{Th(E)}$ is 
\[
A^{\ast}(\I{\mathcal{L}}{Th(E)}) = (H(B) \otimes \wedge (a,b), D) \oplus_{H(B) \otimes \wedge a}(\mathbf{Q} \oplus I_{\mathcal{L}},d).
\]

We attach the top cell inducing Poincaré duality by Whitehead products with respect to some basis of $\pi_{\ast}(\I{\mathcal{L}}{Th(E)}) \otimes \mathbf{Q}$ and denote the resulting space by $\DP{Th(E)}$, its cohomology algebra is then 
\[
H^{\ast}(\DP{Th(E)}) = \frac{\mathbf{Q}[e_{4},e_{6}] \otimes \wedge (e_{7},e_{9},e_{13})}{(e_{4}^{2},e_{6}^{2},e_{4}e_{6},e_{4}e_{7},e_{6}e_{9},e_{7}e_{9},e_{6}e_{7}=e_{4}e_{9}=e_{13})}, \, |e_{i}|=i.
\]

\printbibliography

\end{document}